\documentclass[12pt,reqno,oneside]{amsart}
\usepackage{parskip}
\usepackage[]{babel}
\usepackage[utf8]{inputenc} %Para utilizar tildes y caracteres latinos
\usepackage{amsmath,amsthm,amsbsy}
\usepackage{amsfonts}
\usepackage{yfonts}
\usepackage{mathrsfs}
\usepackage{bbm}
\usepackage{amssymb}
\usepackage[colorlinks=true, menucolor=blue, linkcolor=blue, citecolor=blue]{hyperref}
\usepackage{enumitem}
\usepackage{multicol}
\usepackage{float}
\usepackage{fancyhdr}
\usepackage{layout}
\usepackage{esint}
\usepackage{array}
\usepackage{makecell}
\allowdisplaybreaks

\usepackage[colorinlistoftodos]{todonotes}
\usepackage{orcidlink} %Orcid

\usepackage{mathtools}
\usepackage[foot]{amsaddr}

\theoremstyle{plain}
\newtheorem{theorem}{Theorem}[section]              \newtheorem{proposition}[theorem]{Proposition} \newtheorem{lemma}[theorem]{Lemma}
\newtheorem{corollary}[theorem]{Corollary}

\theoremstyle{definition}
\newtheorem{definition}[theorem]{Definition}
\theoremstyle{remark}
\newtheorem{remark}[theorem]{Remark}
\newtheorem{example}[theorem]{Example}

\newtheorem{assumption}[theorem]{Assumption}

\makeatletter \@addtoreset{equation}{section} \makeatother

\newcommand{\goth}[1]{\mathfrak{#1}} %Gothic letters

\newcommand{\defeq}{:=}

\newcommand{\N}{\mathbb{N}}     % Natural numbers
\newcommand{\R}{\mathbb{R}}     % Real numbers
\newcommand{\Prob}{\mathbb{P}}  % Probability meas
\newcommand{\Exp}{\mathbb{E}}   % Expectation 

\newcommand{\caract}{\mathbbm{1}}
\newcommand{\inner}[2]{\left( #1 \, , \, #2 \right)} % Inner product
\newcommand{\norm}[1]{\left\|#1\right\|}              % Vector norm
\newcommand{\triplet}[3]{\left( #1, #2, #3 \right) }    % General triplet e.g. a probability space
\newcommand{\ProbSpace}{\triplet{\Omega}{\mathcal{F}}{\Prob}}    % Triplet of a Probability Space
\newcommand{\abs}[1]{\left| #1 \right|} % Absolute value  

\newcommand{\quadraVari}[1]{\left\langle  #1  \right\rangle } % quadratic variation symbol

\newcommand{\operQuadraVari}[1]{\left\langle \!\left\langle #1  \right\rangle \!\right\rangle} % operator quadratic variation symbol

\title[It\^{o}'s Formula]{It\^{o}'s Formula for It\^{o} processes defined with respect to a cylindrical martingale-valued measure}

\author{S. Cambronero\,\orcidlink{0000-0001-6758-4942}$^1$}
\author{D. Campos\,\orcidlink{0000-0002-3608-1151}$^2$}
\author{C. A. Fonseca-Mora\,\orcidlink{0000-0002-9280-8212}$^3$}
\author{D. Mena\,\orcidlink{0000-0002-9443-391X}$^4$}
\address{Centro de Investigaci\'{o}n en Matem\'{a}tica Pura y Aplicada \\ Escuela de Matem\'{a}tica, Universidad de Costa Rica}
\email{$^1$ santiago.cambronero@ucr.ac.cr} 
\email{$^2$ josedavid.campos@ucr.ac.cr}
\email{$^3$ christianandres.fonseca@ucr.ac.cr}
\email{$^4$ dario.menaarias@ucr.ac.cr}

\begin{document}
\emergencystretch 3em


\begin{abstract}
Using the authors' recently developed stochastic integration [Stoch PDE: Anal Comp, 2024], we prove an It\^{o} formula  for Hilbert space-valued It\^{o} processes defined with respect to a cylindrical martingale-valued measure. We develop some tools from stochastic analysis, as are the predictable and optional quadratic variation of a stochastic integral, the continuous and purely discontinuous parts of an integral process, and a Riemann representation formula. As an application of our It\^{o} formula, we prove a  Burkholder inequality for the stochastic integral defined with respect to a cylindrical martingale-valued measure.  Finally, we derive It\^{o} formulas for Hilbert space-valued martingale-valued measures and for cylindrical square integrable martingales.
\end{abstract}

\maketitle

%\tableofcontents
\section{Introduction}

It\^{o}'s formula is one of the most important tools in stochastic analysis and has many theoretical and practical applications, for instance, to the resolution of stochastic differential equations. In the real-valued case and for Wiener processes, the formula was first established by K. It\^{o} in \cite{Ito:1951}. An extension to Hilbert space-valued It\^{o} processes, driven by a Wiener process, was introduced by Curtain and Falb \cite{CurtainFalb}. Later, Kunita \cite{Kunita:1970} made a further extension  to Hilbert space-valued It\^{o} processes driven by square integrable martingales.  An even further extension to Hilbert space-valued semimartingales was carried out by M\'{e}tivier in \cite{Metivier}. Some recent works that establish It\^{o} formulas for other types of noises or for certain classes of Banach spaces are \cite{BodoRiedleTylb:2024, BrzezniakVanNeervenVeraarWeis:2008, DaPratoJentzenRockner:2019, DiGirolamiRusso:2014, Fabbri-Russo:2017, GyongyWu:2021, RudigerZiglio:2006, VeraarYaroslavtsev:2016}. 

In this paper, we prove an It\^{o} formula for Hilbert space-valued It\^{o} processes driven by a cylindrical martingale-valued measure. To be precise, let $H$, $G$ and $K$ be separable Hilbert spaces, $\beta:\Omega \times \R_{+} \rightarrow \R$ an adapted c\`{a}dl\`{a}g process, locally of finite variation. We also consider a ring $\mathcal{A}$ of subsets of a topological space $U$ and $(M(t,A): t \geq 0, A \in \mathcal{A})$ a cylindrical martingale-valued measure. That is, for each $A \in \mathcal{A}$, $M(\cdot,A)$ consists on a collection of cylindrical random variables on $H$, and for $t\geq 0$, $M(t,\cdot)$ is an $L^{2}$-valued measure defined on $\mathcal{A}$. Let $X$ be the $G$-valued It\^{o} process:
$$ d X_t = \Psi(t) d\beta(t) + \int_U \Phi(t, u)\, M(dt, du) $$ 
for coefficients $\Psi: \Omega \times \R_{+} \rightarrow G$ and $\Phi: \Omega \times \R_{+} \times U \rightarrow \mathcal{HS}(H,G) $ satisfying suitable integrability conditions, where $\mathcal{HS}(H,G)$ denotes the space of Hilbert-Schmidt operators from $H$ to $G$. 
If  $f:[0, T] \times G \rightarrow K$ is a function of class $C^{1,2}$, we show the following It\^{o} formula:
$$
\begin{aligned}
& f(t, X_t) = f(0, X_{0})+ \int_{0}^{t} f_{t}(s, X_{s-}) d s +\int_{0}^{t}\!\!\int_U  f_{x}(s, X_{s-}) \Phi(s, u)  M(ds, du) \\
& + \int_{0}^{t} f_{x}(s, X_{s-}) \psi(s) d\beta(s)   +\frac{1}{2} \int_{0}^{t} \!\!\int_U 
\mbox{Tr}_{\Phi(s, u) Q_{M^{c}}^{1 / 2}} \left( f_{x x}(s, X_{s-}) \right)
\operQuadraVari{M^{c}}(ds,du) \\
 & +\sum_{0<s \leq t}\left[f\left(s,X_s\right)-f\left(s,X_{s-}\right)-f_x\left(s,X_{s-}\right)\left(\Delta X_{s} \right)\right].
\end{aligned}
$$
The second integral on first line corresponds to a stochastic integral with respect to $M$. The remaining three
integrals are defined almost surely as $K$-valued Bochner integrals with respect to the Lebesgue measure, to $\beta$, and to $\operQuadraVari{M^{c}}$ respectively. The process $\operQuadraVari{M^{c}}$ is the predictable quadratic variation of the cylindrical martingale-valued measure $M^{c}$, corresponding to the continuous part of $M$.

The theory of stochastic integration with respect to a cylindrical martingale-valued measure $M$ was introduced by the authors in \cite{CCFM:SPDE}. A fundamental step in the construction of the stochastic integral was the introduction of the notion of predictable quadratic variation $\operQuadraVari{M}$ for $M$. The definition of such an object is non-trivial, because $M$ is not an $H$-valued martingale measure (as it is assumed, for instance, in \cite{Applebaum:2006}) and hence tools from the theory of $H$-martingales (which can be found for instance in \cite{Metivier}) are not at our disposal. Since $M$ encloses both the concept of cylindrical martingale and of martingale-valued measure, in  \cite{CCFM:SPDE} the authors found it convenient to formulate a definition for $\operQuadraVari{M}$ as a supremum of a family of intensity measures (in the sense introduced by Walsh in \cite{Walsh:1986}). Such a definition extended the work carried out by Veraar and Yaroslatsev in \cite{VeraarYaroslavtsev:2016} to the context of cylindrical martingale-valued measures.

In Section \ref{sectPreliminaries} we recall the main definitions and properties of cylindrical martingale-valued measures and their quadratic variation, as well as the stochastic integral defined with respect to these objects in \cite{CCFM:SPDE}. In Section \ref{sectQuadraVariaStochIntegral}  we extend the theory of stochastic integration, by introducing some fundamental tools we will need for the formulation and proof of our It\^{o} formula. First, in Section \ref{subSecItoIsometry} we prove a conditional version of It\^{o}'s isometry which is later used in Section \ref{subSectPrediQuadVariStochInteg} to compute the predictable quadratic variation of the stochastic integral. We next show, in Section \ref{subSecContAndDisconCMVM}, that under basic conditions a cylindrical martingale-valued measure $M$ can be decomposed as an orthogonal sum of a cylindrical martingale-valued measure with continuous paths $M^{c}$ and a cylindrical martingale-valued measure with purely discontinuous paths $M^{d}$. Both $M^{c}$ and $M^{d}$ have predictable quadratic variations $\operQuadraVari{M^{c}}$ and $\langle\!\langle M^{d} \rangle\! \rangle$ with corresponding (operator-valued) covariance processes $Q^{c}$ and $Q^{d}$. Moreover, we show that the stochastic integral with respect to $M$ can be decomposed as the sum of the stochastic integrals with respect to $M^{c}$ and $M^{d}$. Later, in  Section \ref{subSecOptioQuaVariStochInteg} we show that the optional quadratic variation (as a $G$-valued process) of the stochastic integral with respect to $M$ can be expressed in terms of the quadratic variation of the stochastic integrals with respect to $M^{c}$ and $M^{d}$. 

The core of this paper is Section \ref{sectItoFormula}, where we prove our version of It\^{o} formula. The result is first proved for scalar valued functions (Theorem \ref{theoItoFormulaScalarCase}) and later extended to vector valued functions (Theorem \ref{theoItoFormulaVectorCase}) via duality and the Hahn-Banach theorem. This reduction to the scalar case simplifies some of the arguments based on convergence under the integral sign, as well as the particular form of the residual on Taylor's formula. Since our cylindrical martingale-valued measures may have discontinuous paths, so do the corresponding It\^{o} processes. Therefore, in the proof of Theorem \ref{theoItoFormulaScalarCase}, one requires a careful analysis of the size of the jumps of the It\^{o} processes. This is where the separation of the quadratic variation of the stochastic integral in terms of the quadratic variations of the continuous and purely discontinuous parts plays a major role. 

The usefulness of our It\^{o}'s formula is shown in Section \ref{subSectBurkholder} where we use it to establish a Burkholder inequality for the stochastic integral with respect to a cylindrical martingale-valued measure. We do this by first assuming $M$ has (cylindrically) continuous paths (Theorem \ref{theoBurkholderContinuousCase}) and then, that it has (cylindrically) purely discontinuous paths (Theorem \ref{theoBurkholderDiscontinuousCase}). The general case is obtained by using the decomposition of the stochastic integral with respect to $M$ as the sum of the stochastic integrals with respect to the continuous and purely discontinuous parts of $M$.
Finally, we apply the Burkholder inequality in the case of an $H$-valued c\`{a}dl\`{a}g L\'{e}vy process (Example \ref{examHvaluedLevy}). We finish this section  by establishing a Kunita inequality as a consequence of this example.

Finally, in Section \ref{sectionAppliItoToHValuedMVM} we directly relate our results to the ones appearing in some other literature.  In the first half, we show how our theory can be applied to Hilbert space-valued martingale-valued measures and we derive a corresponding It\^o formula.  To the extent of our knowledge, this result is new.  This can be applied to recover the classical It\^o formula in other works, as for example \cite{DaPratoZabczyk, PeszatZabczyk, MetivierPellaumail}). 

In the second subsection, we look at the particular case of a cylindrical square integrable martingale, both, the continuous paths case (recovering the result of \cite{VeraarYaroslavtsev:2016} for the Hilbert space setting) and the general (including discontinuous paths) case.

%%%%%%%%%%%%%%%%%%%%%%%%%%%%%%%%%%%%%%%%
 
\section{Preliminaries}\label{sectPreliminaries}

Throughout this work, $H$, $G$ and $K$ denote separable Hilbert spaces. We will use the notation $(\cdot,\cdot)_{H}, (\cdot,\cdot)_{G}, (\cdot,\cdot)_{K}$ for the inner products in $H$, $G$ and $K$ respectively. We identify the dual of a Hilbert space with the space itself. For a Hausdorff topological space $U$, its Borel $\sigma$-algebra will be denoted by $\mathcal{B}(U)$. 

We assume that $(\Omega, \mathcal{F}, \Prob)$ is a complete probability space equipped with a filtration $(\mathcal{F}_{t})_{t \in \R_{+}}$ that satisfies the \emph{usual conditions}, that is, it is right continuous and $\mathcal{F}_{0}$ contains all $\Prob$-null sets. The predictable $\sigma$-algebra on $\Omega \times [0,\infty)$ is denoted by $\mathcal{P}$ and for any $T>0$ we denote by $\mathcal{P}_{T}$ the restriction of $\mathcal{P}$ to $\Omega \times [0,T]$ (For further details, see chapters 6 and 7 in \cite{CohenElliott:2015}).

A \emph{cylindrical random variable} on $H$ is a linear and continuous operator $Z:H \rightarrow L^{0} (\Omega, \mathcal{F}, \Prob)$, where $L^{0} (\Omega, \mathcal{F}, \Prob)$ is the space of (equivalence classes of) real-valued random variables, equipped with the topology of convergence in probability. A family of cylindrical random variables $Z=(Z_{t}: t \in [0,T])$ on $X$ is called a \emph{cylindrical stochastic process} on $H$. A cylindrical stochastic process  $M=(M_{t}: t \in [0,T])$ is called a \emph{cylindrical mean-zero square integrable martingale} on $X$ if for every $h \in H$ we have $M(h) \in \mathcal{M}_{T}^{2}$, the linear space of all real-valued, c\`{a}dl\`{a}g, mean-zero, square integrable martingales on the time interval $[0,T]$. This is a Banach space with the norm 
$$
\norm{m}_{\mathcal{M}_{T}^{2}}= \left( \sup_{t \in [0,T]} \Exp \left[  \abs{m(t)}^{2} \right] \right)^{1/2}.
$$

For any two Banach spaces $X$ and $Y$, the Banach space of bounded linear operators from $X$ into $Y$ will be denoted by  $\mathcal{L}(X,Y)$. We will denote the space of real-valued bounded bilinear forms on $X \times Y$ by $\goth{Bil}(X,Y)$. Trace class operators on a Hilbert space $H$ will be denoted by $\mathcal{L}_{1}(H)$.

Let $ {Y}=( {Y}_{t}: t \geq 0)$ be an $H$-valued square integrable martingale. There exists a unique (up to in-distinguishability) predictable, real-valued process $\quadraVari{ Y }_t$ with paths of finite variation such that $\quadraVari{ Y }_0=0$ and $\|Y_t\|_H^2-\quadraVari{ Y}_t$ is a martingale. We define a predictable operator-valued quadratic variation process $\operQuadraVari{Y}_{t}$ of $Y_t$ with values in the space of nonnegative definite trace-class operators on $H$ such that, for all $h_1, h_2 \in H$,
$$(Y_t,h_1)_H (Y_t,h_2)_H-(\operQuadraVari{ Y}_t(h_1),h_2 )_H$$
is a martingale. This process exists and is unique. Moreover $\quadraVari{ Y }_t=\mbox{trace}(\operQuadraVari{ Y}_t)$ (see e.g.  \cite{Metivier}).

There is also a unique real, (up to in-distinguishability) 
regular, right continuous process with paths of finite variation, denoted by $[Y]_t$, such that for every increasing sequence $(\Pi_n)$ of infinite partitions of $\R^+$, where  $\Pi_n=\{0<t^n_0 < t^n_1<\cdots<t^n_k<\cdots\}$, with $\displaystyle \lim_{k \to \infty}t^n_k=+\infty$ and $\displaystyle \lim_{n \to \infty}\sup_{i \geq 0} (t^n_{i+1}-t^n_i)=0$, one has
$$
[Y]_t=\lim_{n} \sum_{i \geq 0} \|Y_{t^n_{i+1}\wedge t}-Y_{t^n_i \wedge t}\|^2_H,
$$
where the limit is taken in the $L^1$-sense (see \cite{Metivier}). There is a connection between the real-valued predictable quadratic variation $\quadraVari{ Y }_t$ and the optional quadratic variation $[Y]_t$, given by the following equality, which is satisfied for every real $t \geq 0$,
$$[Y]_t=\quadraVari{Y^c}_t+ \sum_{s \leq t}\|\Delta Y_s\|_H^2=[Y^c]_t+\sum_{s \leq t}\|\Delta Y_s\|_H^2 \quad a.s,$$
where $Y^c$ is the continuous part of $Y$.

Let $H \widehat{\otimes}_1 H$ denote the completion of the projective tensor product $H \otimes H$ (in fact since $H$ is Hilbert, $H \widehat{\otimes}_1 H$ is equivalent to the nuclear operators defined on $H$). 

As in the case of the real-valued quadratic variation, there exists (see \cite{Metivier}) an $H \widehat{\otimes}_1 H$-valued 
regular right continuous process, which is unique (up to in-distinguishability), denoted by $\left[\!\left[ Y \right]\!\right]_t$, such that for every increasing sequence $(\Pi_n)$ of partitions of $\R^+$, defined as before, one has 
$$
\left[\!\left[  Y\right]\!\right]_t=\lim_{n} \sum_{i \geq 0} (Y_{t^n_{i+1}\wedge t}-Y_{t^n_i \wedge t})^{\otimes_2},
$$
where the limit is in probability. 

Moreover, for every $t \geq 0$,
$$
\left[\!\left[Y\right]\!\right]_t=\operQuadraVari{Y^c}_t+ \sum_{s \leq t} \Delta Y_s^{\otimes_2}=\left[\!\left[Y^c\right]\!\right]_t+\sum_{s \leq t} \Delta Y_s^{\otimes_2} \quad a.s,
$$
where $Y^c$ is the continuous part of $Y$ and the series on the right-hand side are absolutely convergent in $H \widehat{\otimes}_1 H$.

\subsection{Cylindrical martingale-valued measures and quadratic variation}

In this section we review the basic definitions and properties of cylindrical martingale-valued measures and its quadratic variation. We will make emphasis on our assumptions for $M$ and its quadratic variation, as well as the corresponding implications of these assumptions. For further details the reader is referred to \cite{CCFM:SPDE}.

Throughout this work we assume that $M$ is a \emph{cylindrical orthogonal martingale-valued measure} on $H$. To be precise, we fix a Hausdorff topological space $U$, which is Lusin in the sense that it is
homeomorphic to a Borel subset of the line. We consider a ring $\mathcal{A}$ of Borel subsets of $U$. A \emph{cylindrical martingale-valued measure} $M$ is defined as a collection  $(M(t,A): t \geq 0, A \in \mathcal{A})$ of cylindrical random variables on $H$ such that
\begin{enumerate} 
\item For each $A \in \mathcal{A}$, $M(0,A)(h)= 0$ $\Prob$-a.e. for all $h \in H$. \label{timezerocomvmdef}
\item For each $A \in \mathcal{A}$, $M(A) = (M(t,A): t \geq 0)$, is a cylindrical mean-zero square integrable martingale, and for each $t > 0$ and $A \in \mathcal{A}$, the map 
$$ 
M(t,A): H \rightarrow L^{0} \ProbSpace
$$ 
is continuous. \label{cylindricalMartingale}
\item If $t>0$ and $h \in H$, $M(t,\cdot)(h): \mathcal{A} \rightarrow L^{2} \ProbSpace$ is a $\sigma$-finite $L^{2}$-valued measure. As part of the hypothesis, $\mathcal{A}$ contains each $\mathcal{B}(U_n)$, where $(U_n)$ is the sequence of sets increasing to $U$, given by the $\sigma$-finiteness of $M(t,\cdot)h$. This sequence is independent of $t$. (see Definition 3.1 in \cite{CCFM:SPDE}). \label{fixedtimeismeasure}
\end{enumerate}

We further say that $M$ is \emph{orthogonal} if: 
\begin{enumerate}\setcounter{enumi}{3}
\item Given $t>0$ and $h \in H$, $\quadraVari{M(A)(h),M(B)(h)}_{t}=0$ whenever $A,B\in \mathcal{A}$ are disjoint. \label{orthogonality}
\end{enumerate}

To the cylindrical orthogonal martingale-valued measure $M$ we can associate a family of \emph{intensity measures} $(\nu_{h}: h \in H)$ which are defined as follows: given $h \in H$, $\nu_{h}$ is a  random predictable $\sigma$-finite measure  on
$\mathcal{B}(\R+) \otimes \mathcal{B}(U)$  such that for every $t \geq 0$  and $A \in \mathcal{A}$, we have 
\begin{equation}\label{eqDefiIntensityMeasures}
 \Prob \left( \omega \in \Omega:  \nu_{h} (\omega) ([0, t] \times A) = \quadraVari{M(A)(h)}_{t}(\omega) \right) =1.
\end{equation} 
The existence of this family is a consequence of Lemma 5.1 in \cite{CCFM:SPDE}. 

We will assume that $M$ has a quadratic variation, that is, a random measure $\eta: \Omega \rightarrow \mathcal{M}_{+}(\R_{+} \times U, \mathcal{B}(\R_{+}) \otimes \mathcal{B}(U) )$ that satisfies: 
\begin{enumerate}    
    \item Given $t \geq 0$ and $A \in \mathcal{A}$, for $\Prob$-a.e. $\omega \in \Omega$ we have $\eta(\omega)([0,t] \times A) < \infty$.
    \item \label{propertyUpperBoundNuX} $\eta$ is a minimal element  for the collection of all random measures $\zeta: \Omega \rightarrow \mathcal{M}_{+}(\R_{+} \times U, \mathcal{B}(\R_{+}) \otimes \mathcal{B}(U) )$ with the property:   $\forall\, h \in H$ with $\| h \| = 1$,  $\nu_{h} \leq \zeta$.
\end{enumerate} 

We also assume that the family of intensity measures  $(\nu_{h}: h \in H)$ satisfies the \emph{sequential boundedness property}:  if   $(h_{n})$ is dense in the unit sphere, $\norm{h}=1$ and $t>0$, there exists $\Omega_{h} \subseteq \Omega$ with $\Prob(\Omega_{h})=1$, such that for $0 \leq s <t$, $A \in \mathcal{A}$ and $\omega \in \Omega_{h}$,
\begin{equation}
  \label{eqsequentialBoundednessProper}
  \nu_{h}(\omega)((s,t] \times A) \leq  \sup_{n \geq 1} \nu_{h_{n}}(\omega) ((s,t] \times A). 
\end{equation}
This property and Theorem 5.10 in \cite{CCFM:SPDE} guarantee the existence of a $\Prob$-a.e. unique quadratic variation, that we denote by $\operQuadraVari{M}$, and if  $(h_{n})$ is dense in the unit sphere of $H$, then $\operQuadraVari{M}=\sup_{n \geq 1} \nu_{h_{n}}$ $\Prob$-a.e. (here the supremum is taken $\omega$-wise in the sense of supremum of measures as defined in Section  2 of \cite{CCFM:SPDE}). 

The reader must bear in mind that $\operQuadraVari{M}$ is defined as a ``predictable'' quadratic variation for $M$, which can be inferred from the definition of intensity measures in \eqref{eqDefiIntensityMeasures}.  Since we will not introduce an ``optional'' quadratic variation for $M$, we will keep referring to  $\operQuadraVari{M}$ as the quadratic variation of $M$. 

\begin{remark}\label{remaNoOptionalQV}
The existence of the intensity measures for real-valued martingale-valued measures is fundamental for the construction of the quadratic variation. To the extent of our knowledge there is no analogue of such measures for the optional quadratic variation. Indeed, in the proof of existence of intensity measures in  \cite{Walsh:1986} (see also Appendix A in \cite{CCFM:SPDE} for the Hilbert space setting) the main argument relies heavily on the existence of second moments for the martingale-valued measure. This is one of the reasons why in this article we do not consider $p$-integrable martingale-valued measures for $p<2$.      
\end{remark}

We will further assume that the quadratic variation of $M$ satisfies: for $\Prob$-a.e. $\omega \in \Omega$,
\begin{equation}\label{eqBoundedrangequadraticvariation}
    \sup_{A \in \mathcal{A}} \operQuadraVari{M}(\omega)([0,T]\times A) <\infty.
\end{equation}

\begin{remark}\hfill
\label{remarkFinitudQVTU}
\begin{enumerate}
    \item Since $\mathcal{A}$ contains the sets $U_n$ mentioned on \eqref{fixedtimeismeasure} of the definition of $M$, this implies in particular that $\operQuadraVari{M}$ is $\sigma$-finite.
    \item Observe that \eqref{eqBoundedrangequadraticvariation} does not imply that $\operQuadraVari{M}(\omega)([0,T]\times U)$ is finite, since $U$ might not be in $\mathcal{A}$. However, if $U\in\mathcal{A}$ then \eqref{eqBoundedrangequadraticvariation} follows immediately.
\end{enumerate}
\end{remark}

By Theorem 5.20 in \cite{CCFM:SPDE} there exists a random $\goth{Bil}(H,H)$-valued measure $\alpha_{M}$ defined on $\mathcal{B}([0,T]) \otimes \mathcal{B}(U)$ such that for all $h_{1},h_{2} \in H$, $0 \leq s \leq t\leq T$ and $A \in \mathcal{A}$, $\Prob$-a.e. $\omega \in \Omega$,
\begin{equation*}
\label{eqDefiAlphaMCovariation}
\alpha_{M}(\omega)((s,t] \times A)(h_{1},h_{2}) \\ 
=  \quadraVari{M(A)(h_{1}),M(A)(h_{2})}_{t}(\omega)- \quadraVari{M(A)(h_{1}),M(A)(h_{2})}_{s}(\omega).
\end{equation*}
The  random measure $\alpha_{M}$ induces a random measure $\Gamma_{M}$ 
on $\mathcal{B}([0,T]) \otimes \mathcal{B}(U)$ taking values in $\mathcal{L}(H,H)$, called the \emph{quadratic variation operator measure},  by means of the prescription 
\begin{equation*}\label{eqDefiMeasureGammaM}
\inner{\Gamma_{M}(\omega)(C)h_{1}}{h_{2}}=\alpha_{M}(\omega)(C)(h_{1},h_{2}), \quad \forall\,  \, h_{1}, h_{2} \in H, \, C \in \mathcal{B}([0,T]) \otimes \mathcal{B}(U).     
\end{equation*}
By Theorem 5.21 in \cite{CCFM:SPDE}, given $T>0$,  there exists a  process $Q_{M}: \Omega \times [0,T]  \times U \rightarrow \mathcal{L}(H,H)$ such that $\Prob$-a.e. $\omega \in \Omega$
\begin{equation}\label{existenceofQ}
\inner{\Gamma_{M}(\omega)(C)h_{1}}{h_{2}}_H = \int_{C} \inner{Q_{M}(\omega,r,u)h_{1}}{h_{2}}_H \, \operQuadraVari{M} (\omega)(dr,du)    
\end{equation}
for all $h_{1}, h_{2} \in H$, $C \in \mathcal{B}([0,T]) \otimes \mathcal{B}(U)$. Moreover, the following properties hold:
\begin{enumerate}
    \item For every $h_{1}, h_{2} \in H$, the mapping $(\omega,r,u) \mapsto \inner{Q_{M}(\omega,r,u)h_{1}}{h_{2}}$ is predictable, that is, $\mathcal{P}_T \otimes \mathcal{B}(U)$-measurable. \label{Qpredictable}
    \item For $\Prob$-a.e. $\omega \in \Omega$, $Q_{M}(\omega,\cdot,\cdot)$ is positive and symmetric $\operQuadraVari{M}$-a.e. \label{Qpositiveandsymmetric}
    \item For $\Prob$-a.e. $\omega \in \Omega$, $\norm{Q_{M}(\omega,\cdot,\cdot)}_{\mathcal{L}(H,H)}=1$,  $\operQuadraVari{M}$-a.e. \label{normoneoftheoperatorQ}
\end{enumerate}

\subsection{Stochastic integration with respect to cylindrical martingale-valued measures}\label{sectStochIntegra}

In this section we summarize a theory of Hilbert-space valued stochastic integration for operator-valued processes with respect to $M$ (for further details, see \cite{CCFM:SPDE}).

Let $\mathcal{HS}(H,G)$ be the space of Hilbert-Schmidt operators from $H$ to $G$. For the process $Q_{M}: \Omega \times [0,T]  \times U \rightarrow \mathcal{L}(H,H)$ and for each $(\omega, t, u) \in \Omega \times [0,T] \times U$, we consider the space $H_Q=Q^{1/2}(H)$ with inner product given by
$$
(h,g)_{H_Q}=(Q^{-1/2}h, Q^{-1/2}g)_H.
$$
Here $Q^{-1/2}$ is the pseudo inverse of $Q^{1/2}$. This space is isometric to $((\mbox{Ker}(Q^{1/2}))^{\perp},(\cdot, \cdot)_H)$, so if we have a complete orthonormal system $\{h_k\}$ on $(\mbox{Ker}(Q^{1/2}))^{\perp}$, we get a complete orthonormal system $\{Q^{1/2}h_k\}$ on $H_Q$. Even more, that system can be supplemented to a complete orthonormal system on $H$  by elements of $\mbox{Ker}(Q^{1/2})$ and thus the equality
$$
\norm{L}_{\mathcal{HS}(H_{Q},G)} = \norm{L\circ Q^{1/2}}_{\mathcal{HS}(H,G)}
$$
holds for any $L \in \mathcal{HS}\left(H_Q,G\right)$.

Throughout this work, we consider a class of integrands  given by families of operators $\Phi(\omega,t,u),$ indexed by $(\omega,t,u)\in \Omega\times [0,T]\times U$, such that:
\begin{enumerate}
\item \label{domainIntegrands} For each $(\omega,t,u) \in \Omega\times [0,T]\times U$, $\Phi(\omega,t,u)\in \mathcal{HS}(H_{Q_{M}},G)$. 
\item \label{predictabilityIntegrands} For every $h \in H$, the $G$-valued process $\Phi \circ Q^{1/2}_{M}(h)$ is predictable, that is, 
the mapping 
\begin{equation}\label{eqPredictabilityIntegrands}
 (\omega,t,u) \mapsto \Phi(\omega,t,u) \circ Q^{1/2}_{M}(\omega,t,u)(h),  
\end{equation}
 is $\mathcal{P}_{T}\otimes \mathcal{B}(U)/\mathcal{B}(G)$-measurable. 
\item \label{integrabilityIntegrands} $\norm{\Phi}^2_{\Lambda^2(M,T)}$ is finite, where this quantity is defined by
\begin{align}
\label{NewIntegrands}
\norm{\Phi}^2_{\Lambda^2(M,T)} & =
\Exp \int_{[0,T]\times U} \| \Phi(\omega,t,u)\circ Q_{M}^{1/2}(\omega,t,u)\|^2_{\mathcal{HS}(H,G)} \, \operQuadraVari{M}(dt,du) \nonumber\\
& = \int_{\Omega \times [0,T]\times U} \| \Phi(\omega,t,u)\circ Q_{M}^{1/2}(\omega,t,u)\|^2_{\mathcal{HS}(H,G)} \,d\mu_M(\omega, t,u),
\end{align}
\end{enumerate}
where $\mu_M$ is the Dol\'eans measure associated to $M$ on $(\Omega \times [0,T] \times U, \mathcal{P}_T \otimes \mathcal{B}(U))$ (see Definition 6.1 and Lemma 6.2 of \cite{CCFM:SPDE}). This space of integrands is denoted by $\Lambda^2(M,T)$ and it is complete with the Hilbertian seminorm given by (\ref{NewIntegrands}). 

\begin{remark}
\label{remarkFinitudNormaPhi}
    According to Remark \ref{remarkFinitudQVTU}, the quadratic variation $\operQuadraVari{M}$ could be infinite on $[0,T]\times U$, so the finiteness of \eqref{NewIntegrands} relies on the integrand $\Phi$.
\end{remark}

In order to define a stochastic integral for this space, we first consider a class of simple operator-valued maps, denoted by $\mathcal{S}(M,T)$, which is contained in $\Lambda^2(M,T)$. Elements in $\mathcal{S}=\mathcal{S}(M,T)$ can be expressed in the form
\begin{equation}
\label{eqSimpleIntegrands}
\Phi(\omega,r,u)=  \sum_{i=1}^{n}\sum_{j=1}^{m} \caract_{(s_{i}, t_{i}]}(r) \caract_{F_{i}}(\omega) \caract_{A_{j}}(u) S_{i,j},
\end{equation}
for all $r \in [0,T]$, $\omega \in \Omega$, $u \in U$, where $m$, $n \in \N$, and for $i=1, \dots, n$, $j=1, \dots, m$, $0\leq s_{i}<t_{i} \leq T$, $F_{i} \in \mathcal{F}_{s_{i}}$, $A_{j} \in \mathcal{R}$ and $S_{i,j} \in \mathcal{HS}(H,G)$. For $\Phi\in \mathcal{S}(M,T)$ we define its stochastic integral $I(\Phi)$ as the $G$-valued process given by 
\begin{equation}
\label{NewDefIntSimpleIntegrand}
I_t(\Phi) : = \int_0^t\! \! \int_U \Phi(s,u) M(ds,du) = 
\sum_{i=1}^{n} \sum_{j=1}^{m} \caract_{F_i}(\omega) Y_{i,j}(t)
\end{equation}
where $Y_{i,j}$ is the square integrable martingale taking values in $G$ and satisfying (see \cite{AlvaradoFonseca:2021})
$$
\left( Y_{i,j}(t),g \right)_{G} = M\Bigl((s_i \land t, t_i\land t], A_j\Bigr) (S_{ij}^*g).
$$
The map $I:\mathcal{S}(M,T) \to \mathcal{M}^2(G)$ is linear and continuous and satisfies the It\^{o} isometry:
\begin{equation}
\label{eqItoIsometrySimpleIntegrands}
\Exp \left[ \norm{I_{t}(\Phi)}^{2} \right] = \Exp\int_{[0,t] \times U} \norm{ \Phi(\omega, s,u)\circ Q^{1/2}_M(\omega,s,u)}_{\mathcal{HS}(H,G)}^{2}\operQuadraVari{M}(ds,du).
\end{equation}
Since $\mathcal{S}(M,T)$ is dense in $\Lambda^2(M,T)$, the map $I$ has a continuous linear extension $$I: \Lambda^2(M,T) \rightarrow \mathcal{M}_T^2(G),$$ where (\ref{eqItoIsometrySimpleIntegrands}) holds for $\Phi \in \Lambda^2(M,T)$. In addition, $I(\Phi)$ is a continuous process provided that, for each $A \in \mathcal{A}$ and $h \in H$, the real-valued stochastic process $M(\cdot, A)(h)$ is continuous (see Proposition 6.14 in \cite{CCFM:SPDE}).

In the case that (\ref{NewIntegrands}) is replaced by the weaker condition
\begin{equation}
\label{LocallyIntegrable}
\Prob \left(
\int_{[0,T]\times U} \| \Phi(\omega,s,u)\circ Q_M^{1/2}(\omega,s,u)\|^2_{\mathcal{HS}(H,G)} \, \operQuadraVari{M}(ds,du) < \infty \right) = 1,
\end{equation}
we have an extension to locally integrable operator-valued maps. This space is denoted by $\Lambda^2_{\textup{loc}}(M,T)$, which has a complete, metrizable, linear topology (for further details, see \cite{CCFM:SPDE}). Consider $\Phi \in \Lambda^2_{\textup{loc}}(M,T)$ and the sequence of stopping times given by
$$
\tau_n(\omega) : = \inf \left\{  t \in [0,T] : \int_{[0,T]\times U} \| \Phi(\omega,s,u)\circ Q_M^{1/2}(\omega,s,u)\|^2_{\mathcal{HS}(H,G)} \, \operQuadraVari{M}(ds,du) \geq n  \right\}.
$$
Then $(\tau_n : n \in \N)$ is an increasing sequence satisfying $\tau_n \to T$  $\Prob$-a.e. by \eqref{LocallyIntegrable}.  Since this sequence is $\mathcal{F}_t$-adapted, for each $n \in \N$, the function $\Phi^{\tau_n} = \Phi \caract_{[0, \tau_n]}$ belongs to $\Lambda^2(M,T)$.  Thus, we can extend the definition of the integral to $\Phi \in \Lambda^2_{\textup{loc}}(M,T)$ by setting 
\begin{equation}
    I(\Phi) : = I(\Phi^{\tau_n}), \ \mbox{if}\  \int_{[0,T]\times U} \| \Phi(\omega,s,u)\circ Q_M^{1/2}(\omega,s,u)\|^2_{\mathcal{HS}(H,G)} \, d\operQuadraVari{M} \leq n.
\end{equation}
Indeed, Proposition 6.17 in \cite{CCFM:SPDE} (applied to $\sigma = \tau_n \wedge \tau_m$ for $m < n$) guarantees that this definition is consistent. From this localization procedure, $I(\Phi)$ belongs to the linear space of all the $G$-valued locally  zero-mean square integrable martingales, denoted by $\mathcal{M}_T^{2,loc}(G)$, which is equipped with the topology of uniform convergence in probability on compact intervals of time. The map $I:\Lambda^2_{\textup{loc}}(M,T) \rightarrow \mathcal{M}_T^{2,loc}(G)$ is linear and its continuity follows from Proposition 6.19 of \cite{CCFM:SPDE}.

\section{Quadratic variation of the stochastic integral} \label{sectQuadraVariaStochIntegral}

In this section we continue with the treatment of the properties of the stochastic integral with respect to cylindrical martingale-valued measures by calculating its predictable and optional quadratic variations (Theorem \ref{theoQuadraVariaStochIntegral} below). 

From classical theory, one might expect 
    $$
    \left[ I \right]_t = \int_{0}^{t} \!\!\int_U  \norm{ \Phi(r,u)Q_{M}^{1/2}(r,u)}_{\mathcal{HS}(H,G)}^{2} \left[\!\left[ M \right]\!\right] (dr,du) 
    $$
    where $\left[\!\left[ M \right]\!\right] $ is the optional quadratic variation of $M$. However, $\left[\!\left[ M \right]\!\right]$ has not been defined since we would require existence of intensity measures in the sense of \cite{Walsh:1986} (see Remark \ref{remaNoOptionalQV}), but being optional instead of predictable (see \eqref{eqDefiIntensityMeasures}). This is why we must define $[I]$ as in \eqref{eqOptionalQuadraticVariationIntegral} below, separating the continuous and discontinuous parts (section \ref{subSecContAndDisconCMVM}). This adds an extra layer of difficulty in comparison with other works such as \cite{DaPratoZabczyk,PeszatZabczyk}, where the variations can be computed explicitly. Our approach generalizes those results (see Section \ref{sectionAppliItoToHValuedMVM}).

\subsection{Conditional It\^{o} isometry}\label{subSecItoIsometry}

As a preliminary result we will need the following conditional version of It\^{o}'s isometry. 

\begin{theorem}[Conditional It\^{o} isometry]\label{theoConditionalItoIsometry}
For every $\Phi \in \Lambda^{2}(M,T)$,  $0 \leq s < t \leq T$ and $g \in G$, we have
\begin{eqnarray}
 \Exp \left[ \abs{ \inner{I_{T}\left( \, \mathbbm{1}_{(s,t]} \, \Phi \right)}{g}_G }^{2} \, \vline \, \mathcal{F}_{s} \right]  
 & = & \Exp \left[  \int_{s}^{t} \!\!\int_U \norm{Q_M^{1/2}\Phi^*(r,u)g}_H^2\operQuadraVari{M}(dr,du) \,\, \vline \,\, \mathcal{F}_{s} \right] \label{eqCondiItoIsometry}   \\
 & = &\Exp \left[  \int_{s}^{t} \!\!\int_U \inner{ \Phi(r,u) Q_M \Phi(r,u)^{*}g }{g}_G \, \operQuadraVari{M}(dr,du)  \, \vline \, \mathcal{F}_{s} \right].
 \nonumber
\end{eqnarray}
Moreover
\begin{align}
& \hspace{-1cm}\Exp \left[ \abs{ \inner{I_t(\Phi)}{g}_G }^{2} \, \vline \, \mathcal{F}_{s} \right] \nonumber \\ 
 & =   \abs{ \inner{I_s(\Phi)}{g}_G }^{2} + \Exp \left[  \int_{s}^{t} \!\!\int_U \norm{Q_M^{1/2}\Phi^*(r,u)g}_H^2\operQuadraVari{M}(dr,du)  \, \vline \, \mathcal{F}_{s} \right] \label{eqSecondConditionalMomentIntegral} \\
 & =  \abs{ \inner{I_s(\Phi)}{g}_G }^{2} + \Exp \left[  \int_{s}^{t} \!\!\int_U \inner{ \Phi(r,u) Q_M \Phi(r,u)^{*}g }{g}_G \, \operQuadraVari{M}(dr,du)  \, \vline \, \mathcal{F}_{s} \right] \nonumber
\end{align} 
\end{theorem}

\begin{proof}
Let $\Phi \in \Lambda^{2}(M,T)$ and $g \in G$. Define $\Phi_{g}$ as the family given by 
$$
\Phi_{g}(\omega,r,u)h = \inner{\Phi(\omega,r,u)h}{g}_G,\quad h\in H.
$$
In particular we have
$$
\Phi_{g}(\omega,r,u) Q_M^{1/2}(\omega, r, u)h=\inner{\Phi(\omega,r,u)Q_M^{1/2}(\omega, r, u) h}{g}_G, \quad h \in H,
$$
that is, $\Phi_{g}(\omega,r,u)Q_M^{1/2}(\omega, r, u)  \in \mathcal{L}(H,\R) \simeq H$. 

By the corresponding properties of $\Phi$ it is easy to conclude that $\Phi_{g} \in \Lambda^{2}(M,T,H,\R)$, hence, the stochastic integral $(I_{t}(\Phi_{g}): t \in [0,T])$ is a real-valued square integrable martingale. Therefore, by Proposition 6.15 in \cite{CCFM:SPDE} and the It\^{o} isometry we have
\begin{equation}\label{eqWeakItoIsometry}
\Exp \left[ \abs{\inner{I_{T}(\Phi)}{g}_G}^{2} \right] =\Exp \left[ \abs{I_{T}(\Phi_{g})}^{2} \right] = \int_{\Omega} \int_{0}^{T} \!\!\int_U \norm{Q_M^{1/2}\Phi^*(r,u) g}_H^{2} d\mu_M.    
\end{equation}

Let $0 \leq s < t \leq T$ and choose any $F \in \mathcal{F}_{s}$. Replacing $\Phi$ with $\mathbbm{1}_{(s,t]\times F}\Phi$ in \eqref{eqWeakItoIsometry} we obtain
$$\Exp \left[ \mathbbm{1}_{F} \abs{\inner{I_{T}(\mathbbm{1}_{(s,t]}\Phi)}{g}}^{2} \right] = \int_F \int_{s}^{t} \!\!\int_U \norm{Q_M^{1/2} \Phi^*(r,u) g}_H^{2} d\mu_M ,
$$
and since $F \in \mathcal{F}_{s}$ is arbitrary this shows \eqref{eqCondiItoIsometry}. In order to prove \eqref{eqSecondConditionalMomentIntegral} observe that, by Proposition 4.21 in \cite{AlvaradoFonseca:2021} we have $I_{T}(\mathbbm{1}_{(s,t]}\Phi)=\left(I_{t}(\Phi)-I_{s}(\Phi) \right)$.
Therefore
\begin{equation*}
\abs{\inner{I_{T}(\mathbbm{1}_{(s,t]}\Phi)}{g}}^{2}= \abs{\inner{I_{t}(\Phi)}{g}}^{2}-2\inner{I_{t}(\Phi)}{g}\inner{I_{s}(\Phi)}{g}+\abs{\inner{I_{s}(\Phi)}{g}}^{2}. 
\end{equation*}
Taking conditional expectation with respect to $\mathcal{F}_{s}$ and using that $(\inner{I_{r}(\Phi)}{g}: r \in [0,T])$ is a real-valued martingale we conclude that 
$$
\Exp \left[ \abs{\inner{I_{T}(\mathbbm{1}_{(s,t]}\Phi)}{g}_G}^{2} \, \vline \, \mathcal{F}_{s} \right] 
=  \Exp \left[ \abs{\inner{I_{t}(\Phi)}{g}_G}^{2} \, \vline \, \mathcal{F}_{s} \right] - \abs{\inner{I_{s}(\Phi)}{g}_G}^{2}.
$$
From the above equality and  \eqref{eqCondiItoIsometry} we deduce \eqref{eqSecondConditionalMomentIntegral}.
\end{proof}

\begin{corollary}
For every $\Phi \in \Lambda^{2}(M,T)$, $0 \leq s < t \leq T$, and $g_{1}, g_{2} \in G$, we have
\begin{multline}\label{eqInnerProductCondiItoIsometry}
\Exp \left[ \inner{I_t(\Phi)}{g_{1}}_G \inner{I_t(\Phi)}{g_{2}}_G \, \vline \, \mathcal{F}_{s} \right]
=   \inner{I_s(\Phi)}{g_{1}}_G\inner{I_s(\Phi)}{g_{2}}_G \\ + \Exp \left[ \int_{s}^{t} \!\!\int_U \inner{ \Phi(r,u) Q_M \Phi(r,u)^{*}g_{1} }{g_{2}}_G \, \operQuadraVari{M}(dr,du)  \, \vline \, \mathcal{F}_{s} \right]     
\end{multline}
\end{corollary}
\begin{proof}
The result can be obtained from \eqref{eqSecondConditionalMomentIntegral} by the polarization identity. 
\end{proof} 

\subsection{Predictable quadratic variation of the stochastic integral}\label{subSectPrediQuadVariStochInteg}

Both the operator-valued and real-valued (predictable) quadratic variation of the stochastic integral are calculated in the following result.

\begin{theorem}\label{theoQuadraVariaStochIntegral}
Let $\Phi \in \Lambda^{2}(M,T)$ and consider the stochastic integral process  
$$ I_{t} = \int^{t}_{0} \!\!\int_U \Phi (r,u) M (dr, du).$$ 
\begin{enumerate}
    \item The predictable quadratic variation of $I$ is given by 
\begin{equation}\label{eqQuadraticVariationIntegral}
\quadraVari{ I }_t = \int_{0}^{t} \!\!\int_U  \norm{ \Phi(r,u)Q_M^{1/2}(r,u)}_{\mathcal{HS}(H,G)}^{2} \operQuadraVari{M}(dr,du).       
\end{equation}
\item The predictable operator quadratic variation of $I$ is given by \begin{equation}\label{eqOperatorQuadraticVariationIntegral}
\operQuadraVari{I}_{t}   =  \int_{0}^{t} \!\!\int_U  \Phi(r,u) Q_M(r,u) \Phi(r,u)^{*}  \operQuadraVari{M}(dr,du).
\end{equation}
\end{enumerate}
\end{theorem}
\begin{proof}
We first prove  \eqref{eqOperatorQuadraticVariationIntegral}. Since $\Phi \in \Lambda^{2}(M,T)$, it follows that 
$$
\Phi(\omega,r,u) Q_M(\omega,r,u) \Phi(\omega,r,u)^{*} \in \mathcal{L}_{1}(G,G),
$$
for each $r \in [0,T]$, $\omega \in \Omega$, $u \in U$. It is clear that
for every $g_{1}, g_{2} \in G$, the mapping 
$$ (\omega,r,u) \mapsto \inner{\Phi(\omega,r,u) Q_M(\omega,r,u) \Phi(\omega,r,u)^{*}g_{1}}{g_{2}}_G, $$ 
is $\mathcal{P}_{T} \otimes \mathcal{B}(U)/\mathcal{B}(\R_{+})$-measurable. Therefore the $ \mathcal{L}_{1}(G,G)$-valued process
$$(\Phi(\omega,r,u) Q_M(\omega,r,u) \Phi(\omega,r,u)^{*} : r \in [0,T], \omega \in \Omega, u \in U)$$ is weakly predictable, hence (strongly) predictable since the the separability of $G$ implies that of $ \mathcal{L}_{1}(G,G)$.\medskip

Now let $(g_{n}: n \in \N)$ be any orthonormal basis in $G$. For every $t \in [0,T]$ we have 
\begin{align*}
& \hspace{-2cm}\int_{\Omega \times [0,t] \times U} \norm{ \Phi(r,u) Q_M \Phi(r,u)^{*} }_{\mathcal{L}_{1}(G,G)} \, d\mu_M \\
&=  \int_{\Omega \times [0,t] \times U} \sum_{n =1}^{\infty}  \inner{\Phi(r,u) Q_M \Phi(r,u)^{*}g_{n}}{g_{n}} \, d\mu_M \\
&= \int_{\Omega \times [0,t] \times U} \sum_{n =1}^{\infty} \inner{    Q_M^{1/2}\Phi(r,u)^{*}g_{n}}{Q_M^{1/2}\Phi(r,u)^{*} g_{n}} \, d\mu_M \\
&= \int_{\Omega \times [0,t] \times U} \sum_{n =1}^{\infty}  \norm{ Q_M^{1/2}  \Phi(r,u)^{*}g_n }_G^2 \, d\mu_M\\
&= \int_{\Omega \times [0,t] \times U}  \norm{  \Phi(r,u) Q_M^{1/2}}_{\mathcal{HS}(H,G)}^{2} \, d\mu_M<\infty.
\end{align*}

Therefore, for every $t \in [0,T]$, the right-hand side in \eqref{eqOperatorQuadraticVariationIntegral} is a well-defined Bochner integral with values in $G$ for $\Prob$-a.e. $\omega \in \Omega$.  Define 
$$
\displaystyle J_t = \int_{0}^{t}\!\!\int_U  \Phi(r,u) Q_M \Phi(r,u)^{*}   \operQuadraVari{M}(dr,du),
\ t \in [0,T].
$$
We must show that for any $g_{1}, g_{2} \in G$, the real-valued process
\begin{equation}\label{eqMartingaleQuadraticVariationIntegral}
\inner{I_{t}}{g_{1}}_G\inner{I_{t}}{g_{2}}_G- \inner{ J_t g_{1}}{g_{2}}_G  
\end{equation}
is a martingale. 

In fact, let  $0 \leq s < t \leq T$ and $g_{1}, g_{2} \in G$. Taking conditional expectation in \eqref{eqMartingaleQuadraticVariationIntegral} with respect to $\mathcal{F}_{s}$, and then applying \eqref{eqInnerProductCondiItoIsometry}, we obtain
$$
\Exp \left(\inner{I_{t}}{g_{1}}_G \inner{I_{t}}{g_{2}}_G - \inner{J_t g_{1}}{g_{2}}_G  \, \vline \, \mathcal{F}_{s} \right)
= \inner{I_s}{g_{1}}_G \inner{I_s}{g_{2}}_G - \inner{J_s g_{1} }{g_{2}}_G.  
$$
This shows that the real-valued process given by \eqref{eqMartingaleQuadraticVariationIntegral} is a martingale. Thus \eqref{eqOperatorQuadraticVariationIntegral} follows.\medskip

To prove \eqref{eqQuadraticVariationIntegral}, observe that since $\Phi \in \Lambda^{2}(M,T)$ the integral in the right-hand side of \eqref{eqQuadraticVariationIntegral} is a well-defined Lebesgue integral for $\Prob$-a.e. $\omega \in \Omega$ (see also Remark \ref{remarkFinitudQVTU}).

Let $(g_{n}: n \in \N)$ be any orthonormal basis in $G$. Then, by \eqref{eqOperatorQuadraticVariationIntegral} we  have
\begin{eqnarray*}
\quadraVari{ I }_{t} 
& = & \mbox{trace}\left( \operQuadraVari{ I}_{t}  \right) = \sum_{n =1}^{\infty} \inner{ J_tg_{n}}{g_{n}}_G \\
& = & \int_{0}^{t} \!\!\int_U \sum_{n =1}^{\infty}\norm{ Q_M^{1/2} \Phi(r,u)^{*}g_n }_G^2 \operQuadraVari{M}(dr,du) \\
& = & \int_{0}^{t} \!\!\int_U  \norm{ \Phi(r,u)Q_M^{1/2}(r,u)}_{\mathcal{HS}(H,G)}^{2} \operQuadraVari{M}(dr,du).
\end{eqnarray*}
Which proves the desired result.
\end{proof}

The next result delves deeper into the relation between our theory of stochastic integration for cylindrical martingale-valued measures and the corresponding theory of stochastic integration for Hilbert space-valued square integrable martingales (defined for example in \cite{MetivierPellaumail, Metivier}). 
In particular,  we show an associativity property for the stochastic integral; this result will play an important role in the proof of our It\^{o} formula.   

\begin{theorem}\label{theoAssociatiStochIntegral}
For $\Phi \in \Lambda^2(M, T;H,G)$, consider the $G$-valued stochastic integral
$$
Y_{\cdot}=\int_0^{\cdot}\!\! \int_U \Phi(s, u) M(ds,du) \in \mathcal{M}^2(G)
$$
Denote by $\langle Y \rangle$ the predictable quadratic variation of $Y$ and by $Q_Y$ its corresponding (co)variance operator. Let $\Psi: \Omega \times [0,T] \rightarrow \mathcal{L}(G,K)$ with the property that for every $g \in G$ the $K$-valued process $\Psi \circ Q_Y^{1 / 2} (g)$ is predictable.  Then, $\Psi$ is stochasticaly integrable with respect to $Y$, that is,
$$
\mathbb{E} \int_0^T\left\|\Psi(s) Q_Y^{1 / 2}\right\|_{\mathcal{HS}(G,K)}^2 d\langle Y \rangle_s < \infty,
$$
if and only if $\Psi \circ \Phi \in  \Lambda^2(M, T;H,K)$.  Moreover
\begin{equation}
\label{eqAssociativityOfStochIntegral}
    \int_0^t \Psi(s)dY_s = \int_0^t \!\!\int_U \Psi(s) \circ \Phi(s,u) M(ds,du).
\end{equation}
\end{theorem}
\begin{proof} 
First, note that by \eqref{eqQuadraticVariationIntegral}
$$
\quadraVari{Y}_t = \int_0^t \!\! \int_U \| \Phi(s,u) Q_M^{1/2} \|^2_{\mathcal{HS}(H,G)} \operQuadraVari{M}(ds,du),
$$
and an easy computation shows that $\displaystyle Q_Y = \frac{ \Phi Q_M \Phi^* }{ \| \Phi \circ Q^{1/2}   \|^2_{\mathcal{HS}(H,G)} }$. Then we have
\begin{align*}
\mathbb{E} \int_0^T \norm{\Psi(r)Q_Y^{1/2}}_{\mathcal{HS}(G,K)}^2 d\langle Y \rangle_s & = \mathbb{E} \int_0^T \operatorname{Tr}(\Psi \circ Q_Y \circ \Psi^*) d\langle Y \rangle_s \\
& = \mathbb{E} \int_0^T \operatorname{Tr}(\Psi \circ \Phi \circ Q_M \circ \Phi^* \circ \Psi^*) \operQuadraVari{M}(ds,du) \\
& = \mathbb{E} \int_0^T \norm{\Psi(s) \Phi(s,u)Q_M^{1/2}}_{\mathcal{HS}(H,K)}^2 \operQuadraVari{M}(ds,du).
\end{align*}

This proves that $\Psi$ is integrable with respect to $Y$ if an only if $\Psi \circ \Phi \in  \Lambda^2(M, T;H,K)$.  
To prove \eqref{eqAssociativityOfStochIntegral}, we first look at a general integrand $\Phi$ and the elementary integrand: 
$\Psi(\omega,r) = \caract_F (\omega) \caract_{(s,t]}(r) R$, where $0\leq s<t \leq T$, $F \in \mathcal{F}_{s}$ and $R\in \mathcal{L}(G,K)$. By Proposition 6.16 in \cite{CCFM:SPDE} we have 
\begin{eqnarray*}
    \int_0^T \Psi(r) dY_r & = & \caract_{F} R(Y_t-Y_s) = \caract_{F} R\left( \int_s^t \!\!\int_U \Phi(r,u) M(dr,du) \right) \\
    & = & \int_0^T \!\! \int_U \, \caract_F(\omega) \caract_{(s,t]} R\circ \Phi(r,s) \,  M(dr,du) \\
    & = & \int_0^T \!\! \int_U \, \Psi \circ \Phi(r,s) \, M(dr,du).
\end{eqnarray*} 
By linearity, this extends to simple functions, and by density to all $\Psi$ (see Proposition 14.5 in \cite{MetivierPellaumail}).
\end{proof}

\subsection{Continuous and purely discontinuous part of a cylindrical martingale-valued measure}\label{subSecContAndDisconCMVM}

Given a real-valued martingale $m$, denote respectively by $m^{c}$ and $m^{d}$ the continuous martingale part and the purely discontinuous martingale part of $m$. 
We can decompose $M$ as the sum of a continuous-path cylindrical  martingale-valued measure $M^{c}$ and a purely discontinuous one $M^d$, by defining for every $t\geq 0$, $A \in \mathcal{A}$ and $h \in H$,  
$$
M^{c}(t,A)(h)=[M(t,A)(h)]^{c}, \quad M^{d}(t,A)(h)=[M(t,A)(h)]^{d}.
$$
Let's prove that these two processes are in fact cylindrical martingale-valued measures. By definition of $M^c$ and $M^d$, for fixed $A \in \mathcal{A}$, $\alpha \in \mathbb{R}$, $h_1, h_2 \in H$ we have
\begin{align*}
M(t, A)(\alpha h_1 + h_2) & = \alpha M^c(t, A)(\alpha h_1 + h_2) +  M^d(t, A)(\alpha h_1 + h_2).
\end{align*}
On the other hand, by linearity of $M$,
\begin{align*}
M(t, A)(\alpha h_1 + h_2) & = \alpha M(t, A)(h_1) + M(t, A)(h_2) \\
& = \alpha [M^c(t, A)(h_1) + M^d(t, A)(h_1)] + [M^c(t, A)(h_2) + M^d(t, A)(h_2)] \\
& = [\alpha M^c(t, A)(h_1) + M^c(t, A)(h_2)] + [\alpha M^d(t, A)(h_1) + M^d(t, A)(h_2)]. 
\end{align*}
Since the first two terms on the right hand side are continuous, and the other two are purely discontinuous,  the uniqueness of the representation gives us
$$
M^c(t, A)(\alpha h_1 + h_2) = \alpha M^c(t, A)(h_1) + M^c(t, A)(h_2),
$$
and likewise for $M^d$.  This proves that for fixed $A \in \mathcal{A}$ and $t \geq 0$, $M^c(t, A)$ and $M^d(t, A)$ are linear continuous maps from $H$  to $L^{2} \ProbSpace$.
For $t > 0$ and $h \in H$, the $\sigma$-finiteness of $M^c(t,\cdot)(h)$ and $M^d(t,\cdot)(h)$ as $L^2$-valued measures follow directly from their definition, and the corresponding $\sigma$-finiteness as real-valued martingales.

\begin{remark}
    The orthogonality of $M$ does not necessarily imply that of $M^c$ and $M^d$, even in simple real valued cases.  For example, take $U =\{ 0, 1 \}$, and define $M(t, \{0\}) = W_t + \widetilde{N}_t$ and $M(t, \{1\}) = W_t - \widetilde{N}_t$, where $W$ is a standard Browninan motion and $\widetilde{N}_t$ is a compensated Poisson process with intensity $1$. Observe that in this case, $$M^c(t, A) = W_t, \qquad M^d(t, A) = \begin{cases}
    \widetilde{N}_t, & A=\{0\}, \\
    -\widetilde{N}_t, & A=\{1\}.
    \end{cases}$$
    Easy computations show that $\quadraVari{ M(\{0\}), M(\{1\}) }_t = 0$ (and hence is orthogonal), however
    $$
    \quadraVari{ M^c(\{0\}), M^c(\{1\}) }_t = \quadraVari{ M^d(\{0\}), M^d(\{1\}) }_t = t. 
    $$
\end{remark}

For further developments we will require that both $M^{c}$ and $M^{d}$ are orthogonal, and that each one has a unique predictable quadratic variation. 

\begin{assumption}
From now on we assume the following:
\begin{enumerate}
    \item Each one of the cylindrical martingale-valued measures $M^{c}$  and $M^{d}$ is orthogonal. 
    \item The family of intensity measures $(\nu^{c}_{h}: h \in H)$ of $M^{c}$ and $(\nu^{d}_{h}: h \in H)$ of $M^{d}$  satisfies the sequential boundedness property.
\end{enumerate}
\end{assumption}

Observe that the first assumption guarantees the existence of the intensity measures $(\nu^{c}_{h}: h \in H)$ and $(\nu^{d}_{h}: h \in H)$  of $M^{c}$ and $M^{d}$ respectively (see Lemma 5.1 in \cite{CCFM:SPDE}). 

We must show that $M^{c}$ possesses a unique predictable quadratic variation and that $\alpha_{M^{c}}$, $\Gamma_{M^{c}}$ and $Q_{M^{c}}$ exist. First, for every $h \in H$, $0 \leq s <t$ and $A \in \mathcal{A}$, we have $\Prob$-a.e. 
$$\nu^{c}_{h}((s,t] \times A)=\quadraVari{M^{c}(A)(h)}_{t}-\quadraVari{M^{c}(A)(h)}_{s}   \leq \quadraVari{M(A)(h)}_{t}-\quadraVari{M(A)(h)}_{s} = \nu_{h}((s,t] \times A).   $$
Therefore,  for every $h \in H$ by Theorem 3.9 in \cite{CCFM:SPDE} we have $\nu_h^c \leq \nu_h$, more precisely
$$ \Prob \left( \nu^{c}_{h}(C) \leq \nu_{h}(C)  : C \in \mathcal{B}(\R_{+}) \otimes \mathcal{B}(U) \right)=1. $$

Now let $(h_{n}:n \in \N)$ be a dense subset in the unit sphere of $H$. Then $\Prob$-a.e for every $t>0$ and $A \in \mathcal{A}$, by Proposition 5.4 in  \cite{CCFM:SPDE} we have
$$\sup_{n \geq 1} \nu^{c}_{h_{n}} ([0,t]\times A)\leq \sup_{n \geq 1} \nu_{h_{n}}([0,t] \times A) \leq \operQuadraVari{M}([0,t] \times A)<\infty.$$
By Theorem 5.8 in  \cite{CCFM:SPDE}  (or rather its proof) $M^{c}$ has a unique predictable quadratic variation $\operQuadraVari{M^{c}}$ which satisfies $\operQuadraVari{M^{c}}=\sup_{n \geq 1} \nu^{c}_{h_n}$ $\Prob$-a.e. (in particular, $\operQuadraVari{M^{c}} \leq \operQuadraVari{M}$ $\Prob$-a.e.). 
Furthermore, by \eqref{eqBoundedrangequadraticvariation} we have 
for $\Prob$-a.e. $\omega \in \Omega$,
$$ \sup_{A \in \mathcal{A}} \operQuadraVari{M^{c}}(\omega)([0,T]\times A) \leq \sup_{A \in \mathcal{A}} \operQuadraVari{M}(\omega)([0,T]\times A) <\infty.
 $$
Therefore, we can guarantee the existence of the  $\goth{Bil}(H,H)$-valued measure $\alpha_{M^{c}}$, the $\mathcal{L}(H,H)$-valued random measure $\Gamma_{M^{c}}$ and the process $Q_{M^{c}}: \Omega \times [0,T]  \times U \rightarrow \mathcal{L}(H,H)$ satisfying \eqref{existenceofQ}. 

The above arguments are also valid for  $M^{d}$, therefore 
$M^{d}$ has a unique predictable quadratic variation $\langle\!\langle M^d \rangle\!\rangle =\sup_{n \geq 1} \nu^{d}_{h}$ $\Prob$-a.e., with   $\langle\!\langle M^d \rangle\!\rangle \leq \operQuadraVari{M}$ $\Prob$-a.e., the $\goth{Bil}(H,H)$-valued measure $\alpha_{M^{d}}$, the $\mathcal{L}(H,H)$-valued random measure $\Gamma_{M^{d}}$ and the process $Q_{M^{d}}: \Omega \times [0,T]  \times U \rightarrow \mathcal{L}(H,H)$ satisfying \eqref{existenceofQ}. 

The following result clarifies the relation between $Q_{M}$, $Q_{M^{c}}$ and $Q_{M^{d}}$. 

\begin{lemma}\label{lemmaQMCombOfQMCAndQMD}
$\Prob$-a.e. we have 
\begin{equation}\label{eqQMCombOfQMCAndQMD}
   Q_{M}= Q_{M^{c}} \frac{d \operQuadraVari{M^{c}}}{d \operQuadraVari{M}}+ Q_{M^{d}} \frac{d \langle\!\langle M^d \rangle\!\rangle}{d \operQuadraVari{M}}. 
\end{equation}
\end{lemma}
\begin{proof}
By \eqref{existenceofQ} and the orthogonality of both continuous and purely discontinuous parts we have, for $\Prob$-a.e. $\omega \in \Omega$, all $t \geq 0$, $A \in \mathcal{A}$ and $h_{1}, h_{2} \in H$, 
\begin{eqnarray*}
\int_{[0,t]\times A} \inner{Q_{M}h_{1}}{h_{2}} \, \operQuadraVari{M}(\omega)(dr,du)
& = &  \quadraVari{M(A)h_{1},M(A)h_{2}}_{t} \medskip \\
& = & \quadraVari{M^{c}(A)h_{1},M^{c}(A)h_{2}}_{t} + \quadraVari{M^{d}(A)h_{1},M^{d}(A)h_{2}}_{t} \\
& = & \int_{[0,t]\times A} \inner{Q_{M^{c}}h_{1}}{h_{2}} \, \operQuadraVari{M^{c}}(dr,du) \\
 & & + \int_{[0,t]\times A} \inner{Q_{M^{d}}h_{1}}{h_{2}} \operQuadraVari{M^d}(dr,du). 
\end{eqnarray*}
By Theorem 3.9 in \cite{CCFM:SPDE}, this identity extends to 
\begin{eqnarray}
\int_{C}\left(Q_{M} h_{1}, h_{2}\right)_{H} \operQuadraVari{M}\left(dr, du \right) 
& = &  \int_{C}\left(Q_{M^c} h_{1}, h_{2}\right)_{H} \operQuadraVari{M^c}\left(dr, du \right) \nonumber \\
& {} & + \int_{C}\left(Q_{M^d} h_{1}, h_{2}\right)_{H}  \langle\!\langle M^d \rangle\!\rangle (dr, du) \label{eqIntegQMequalQMCplusQMd} 
\end{eqnarray}
for  all $C \in \mathcal{B}(\R_{+}) \times \mathcal{B}(U)$. 
Since $\max\{\operQuadraVari{M^{c}}, \langle\!\langle M^d \rangle\!\rangle\} \leq \operQuadraVari{M}$ $\Prob$-a.e., this gives \eqref{eqQMCombOfQMCAndQMD}. 
\end{proof}

The identity \eqref{eqQMCombOfQMCAndQMD} can be used to say something about the spaces $H_{Q_{M}}$, $H_{Q_{M^{c}}}$ and $H_{Q_{M^{d}}}$. In fact, by \eqref{eqQMCombOfQMCAndQMD}, for each $h\in H$ we have $\Prob$-a.e. 
$$(Q_Mh,h)_{H}=(Q_{M^c}h,h)_{H}\frac{d \operQuadraVari{M^{c}}}{d \operQuadraVari{M}}+(Q_{M^d}h,h)_{H}\frac{d \langle\!\langle M^d \rangle\!\rangle}{d \operQuadraVari{M}}. $$
Equivalently,
\begin{equation}\label{eqSquareQMCombOfSquareQMCAndQMD}
\|Q_M^{1/2}h\|^2=\|Q_{M^c}^{1/2}h\|^2 \frac{d \operQuadraVari{M^{c}}}{d \operQuadraVari{M}}+\|Q_{M^d}^{1/2}h\|^2\frac{d \langle\!\langle M^d \rangle\!\rangle}{d \operQuadraVari{M}}\quad \Prob \text{-a.e.}    
\end{equation}
Therefore, $\Prob$-a.e. we have  
    $${\mbox Ker}\,Q_{M^c}^{1/2}\cap {\mbox Ker}\,Q_{M^d}^{1/2} \subseteq {\mbox Ker}\,Q_M^{1/2}.$$
Now, since the space $H_{Q_M}$ is isometric to $((\mbox{Ker}(Q_{M}^{1/2}))^{\perp},(\cdot, \cdot)_H)$ (the same holds for $Q_{M^c}$ and $Q_{M^d}$), we conclude that 
$$H_{Q_M} \simeq (\mbox{Ker}(Q_{M}^{1/2}))^{\perp}  \subseteq  \overline{ (\mbox{Ker}(Q_{M^{c}}^{1/2}))^{\perp} + (\mbox{Ker}(Q_{M^{d}}^{1/2}))^{\perp} } \simeq  \overline{H_{Q_{M^c}} + H_{Q_{M^d}}}.$$    

Unfortunately the information provided by \eqref{eqSquareQMCombOfSquareQMCAndQMD} is not strong enough to conclude that either $H_{Q_M} \subseteq  H_{Q_{M^c}}$ or $H_{Q_M} \subseteq H_{Q_{M^d}}$. This complicate things when trying to compare the stochastic integrals that correspond to integrands which are integrable at the same time with respect to $M$, $M^{c}$ and $M^{d}$. In particular, the problem arises since the domain of the integrands (see \ref{domainIntegrands} in Section \ref{sectStochIntegra}) for $M$ (respectively for $M^{c}$ or $M^{d}$) is only required to contain the space $H_{Q_{M}}$ (respectively the space $H_{Q_{M^c}}$ or $H_{Q_{M^d}}$); that is, we are working with possibly unbounded operators. For further developments and in order to handle domains of our integrands, we make the following assumption:

\begin{assumption}\label{assumptionDomainIntegrands}
For $\Phi \in \Lambda^{2}(M,T)$ we replace \ref{domainIntegrands} in Section \ref{sectStochIntegra}  by the assumption $\Phi(\omega,t,u) \in \mathcal{HS}(H,G)$ for every $(\omega,t,u) \in \Omega \times [0,T] \times U$, and we replace \ref{predictabilityIntegrands} by the assumption that  $(\omega,t,u) \mapsto \Phi(\omega,t,u) h$ is $\mathcal{P}_{T}\otimes \mathcal{B}(U)/\mathcal{B}(G)$-measurable.     
\end{assumption}

Our course the assumption above implies that  $\Phi(\omega,t,u)\in \mathcal{HS}(H_{Q_{M}},G)$ and  $(\omega,t,u) \mapsto \Phi(\omega,t,u) \circ Q^{1/2}_{M}(\omega,t,u)(h)$ is $\mathcal{P}_{T}\otimes \mathcal{B}(U)/\mathcal{B}(G)$-measurable, so the same applies for $M^{c}$ and $M^{d}$.  As the next result shows, integrability with respect to $M$ implies integrability with respect to $M^{c}$ and $M^{d}$.

\begin{lemma}\label{lemmaNormIntegradsMisASumMcAndMd}
For $\Phi \in \Lambda^{2}(M, T)$ we have
\begin{flalign}
& \mathbb{E} \int_{0}^{T}\!\! \int_{U}\left\|\Phi(r, u) Q_{M}^{1/2}\right\|_{\mathcal{HS}(H, G)}^{2} \operQuadraVari{M}(d r, d u) \nonumber \\ 
&= \mathbb{E} \int_{0}^{T}\!\! \int_{U}\left\|\Phi(r, u) Q_{M^{c}}^{1/2}\right\|_{\mathcal{HS}(H, G)}^{2} \operQuadraVari{M^{c}}(dr,du) \nonumber \\
& +\mathbb{E} \int_{0}^{T}\!\! \int_{U}\left\|\Phi(r, u) Q_{M^{d}}^{1/2}\right\|_{\mathcal{HS}(H, G)}^{2} \operQuadraVari{M^d}(dr,du). \label{eqNormIntegradsMisASumMcAndMd}
\end{flalign}
In particular, $\Phi \in \Lambda^{2}(M^{c}, T)$ and $\Phi \in \Lambda^{2}(M^{d}, T)$. 
\end{lemma} 
\begin{proof}
By \eqref{eqSquareQMCombOfSquareQMCAndQMD}, for any given $(\omega,r, u)$ and $g\in G$ we have, $\Prob$-a.e.
\begin{equation*}
\|Q_M^{1/2}\Phi(r, u)^{*}g\|^2=\|Q_{M^c}^{1/2}\Phi(r, u)^{*} g\|^2 \frac{d \operQuadraVari{M^{c}}}{d \operQuadraVari{M}}+\|Q_{M^d}^{1/2}\Phi(r, u)^{*}g\|^2\frac{d \langle\!\langle M^d \rangle\!\rangle}{d \operQuadraVari{M}}   
\end{equation*}
Then, by considering an orthonormal basis $(g_{n}:n\in \N)$ in $G$, replacing $g$ by $g_{n}$ in the above equality and adding over $n \in \N$, we obtain  
$$ \norm{Q_M^{1/2}\Phi(r, u)^{*}}^2_{\mathcal{HS}(G,H)}
=\norm{Q_{M^c}^{1/2} \Phi(r, u)^{*}}^2_{\mathcal{HS}(G,H)} \frac{d \operQuadraVari{M^{c}}}{d \operQuadraVari{M}}
+\norm{Q_{M^d}^{1/2}  \Phi(r, u)^{*}}^2_{\mathcal{HS}(G,H)} \frac{d \langle\!\langle M^d \rangle\!\rangle}{d \operQuadraVari{M}}.    
 $$
Equivalently, 
$$ \norm{\Phi(r, u)Q_M^{1/2}}^2_{\mathcal{HS}(H,G)}
=\norm{ \Phi(r, u) Q_{M^c}^{1/2}}^2_{\mathcal{HS}(H,G)} \frac{d \operQuadraVari{M^{c}}}{d \operQuadraVari{M}}
+\norm{\Phi(r, u) Q_{M^d}^{1/2}  }^2_{\mathcal{HS}(G,H)} \frac{d \langle\!\langle M^d \rangle\!\rangle}{d \operQuadraVari{M}}.    
 $$
The result follows by integrating both sides with respect to $\operQuadraVari{M}(dr,du)$.
\end{proof}

The following result explains the connection between the stochastic integral with respect to $M$ and the corresponding integrals with respect to $M^{c}$ and $M^{d}$. 

\begin{theorem}
Let $\Phi \in \Lambda^{2, \textup loc}(T, M)$. Then $\Phi  \in \Lambda^{2, \textup loc}(T, M^{c}) \cap \Lambda^{2, \textup loc}(T, M^{d})$ and for every $0 \leq t \leq T$
\begin{equation}\label{eqDecompStochIntegral}
I_t:=\int_0^t \!\! \int_U \Phi(s, u)\, M(ds, du)= \int_0^t \!\! \int_U \Phi(s, u) \, M^{c}(ds, du)+\int_0^t \!\! \int_U \Phi(s, u) \, M^{d}(ds, du).     
\end{equation}
Besides
\begin{equation}
\label{continuousintegralpart}
I_t^c = \int_0^t \!\! \int_U \Phi(s, u)\, M^c(ds, du), \qquad
I_t^d = \int_0^t \!\! \int_U \Phi(s, u)\, M^d(ds, du). 
\end{equation}
\end{theorem}
\begin{proof}
The integrability of $\Phi$ with respect to $M^{c}$ and $M^{d}$ is guaranteed by Lemma \ref{lemmaNormIntegradsMisASumMcAndMd}. 

Consider a simple function $\Phi$ of the form
\begin{equation}\label{eqSimpleFormIntegrand}
\Phi(\omega, r ,u) = \mathbf{1}_{]s_1, s_2]}(r) \mathbf{1}_{F_{s_1}} (\omega) \mathbf{1}_A(u)\, S,
\end{equation}
where $S \in \mathcal{HS} (H,G)$. For $g \in G$, we have
\begin{align}
& \inner{I_t}{g}_{G} = \quad \mathbf{1}_{F_{s_1}} (\omega) M \Bigl((s_1 \wedge t, s_2 \wedge t], A\Bigr) (S^{*}g)\nonumber\\
&=\mathbf{1}_{F_{s_1}} (\omega) M^c \Bigl((s_1 \wedge t, s_2 \wedge t], A\Bigr) ( S^{*}g)+\mathbf{1}_{F_{s_1}} (\omega) M^d \Bigl((s_1 \wedge t, s_2 \wedge t], A\Bigr) (S^{*}g)\nonumber
\end{align}
and hence
\begin{equation}
\inner{I_t}{g}_{G} =\inner{\int_0^t \!\! \int_U \Phi(s, u)\,  M^c(ds, du)}{g}_{G} +\inner{\int_0^t \!\! \int_U \Phi(s, u) \, M^d(ds, du)}{g }_{G}. \label{eqDecompStochIntegralSimpleProcesses}
\end{equation}

% Since $g$ is arbitrary, we have shown that    \eqref{eqDecompStochIntegral}  holds for $\Phi$ of the form \eqref{eqSimpleFormIntegrand}. By a standard density argument one can show that \eqref{eqDecompStochIntegral}  holds for  $\Phi \in \Lambda^{2, \textup loc}(T, M)$. 

On the other hand, by the theory of Hilbert space-valued martingales, 
\begin{equation}
\label{eqOptionalQuadraticVariationIntegralduality2}
\inner{I_t}{g}=\inner{I^c_t}{g}+\inner{I^d_t}{g}.
\end{equation}
Hence, if we use \eqref{eqDecompStochIntegralSimpleProcesses} and (\ref{eqOptionalQuadraticVariationIntegralduality2}), the equality
$$
\inner{I^c_t}{g}_G-\inner{\int_0^t \!\! \int_U \Phi(s, u)\, M^c(ds, du)}{g}_{G} =\inner{\int_0^t \!\! \int_U \Phi(s, u)\, M^d(ds, du)}{g}_{G} -\inner{I^d_t}{g}_G
$$
follows easily. But, this is an equality between a continuous square integrable martingale and a purely discontinuous square integrable martingale. Hence, both sides must be equal to $0$. The theorem follows in the case that $\Phi$ is the form of (\ref{eqSimpleFormIntegrand}), by a direct application of the Hahn-Banach theorem. A density argument completes the proof (see Theorem 6.8 in \cite{CCFM:SPDE}).
\end{proof}

\subsection{Optional quadratic variation of the stochastic integral}\label{subSecOptioQuaVariStochInteg}

As it is well known in other contexts, the optional quadratic variation plays a fundamental role in the statement and proof of It\^{o} formulas. In this section we study the optional quadratic variation of the stochastic integral. In particular, we show that it can be decomposed according to the continuous and discontinuous part of a cylindrical martingale-valued measure. 

First, we need the following terminology on the trace. For $\zeta \in \goth{Bil}(G,G;K)$ and $S,R \in \mathcal{HS}(H,G)$ we use $\mbox{Tr}_{S,R}(\zeta)$  (following \cite{BrzezniakVanNeervenVeraarWeis:2008}) to denote the \emph{trace} defined by
$$ \mbox{Tr}_{S,R}(\zeta)=\sum_{j=1}^{\infty} \zeta(S h_{j},R h_{j}), $$

where $(h_{j})_{j \in \N}$ is an orthonormal basis in $H$. Being a trace, the above quantity is independent of the orthonormal basis used in its calculation. If $S=R$ we write $\mbox{Tr}_{S}(\zeta)$. Moreover, 
$$\norm{\mbox{Tr}_{S,R}(\zeta)}_{K} \leq \norm{\zeta}_{\goth{Bil}(G,G;K)} \norm{S}_{\mathcal{HS}(H,G)} \norm{R}_{\mathcal{HS}(H,G)}.$$

The main result of this section is the following.

\begin{proposition} \label{propRiemannQuadraticVariation} 
Let $\Phi \in \Lambda^{2, \textup loc}(T, M)$ and  consider the stochastic integral process
$$ I_t = \int_0^t \!\! \int_U \Phi(s, u)\, M(ds, du).$$ 
\begin{enumerate}
\item The optional quadratic variation of $I$ is given by 
\begin{equation}\label{eqOptionalQuadraticVariationIntegral}
\left[ I \right]_t = \int_{0}^{t} \!\!\int_U  \norm{ \Phi(r,u)Q_{M^{c}}^{1/2}(r,u)}_{\mathcal{HS}(H,G)}^{2} \operQuadraVari{M^{c}}(dr,du) + \sum_{0<s \leq t} \norm{\Delta I_{s}}_{G}^{2}.  
\end{equation}
\item The optional operator quadratic variation of $I$ is given by 
\begin{equation}\label{eqOperatorOptionalQuadraticVariationIntegral}
\left[\!\left[ I \right]\!\right]_t   =  \int_{0}^{t} \!\!\int_U  \Phi(r,u) Q_{M^{c}}(r,u) \Phi(r,u)^{*}  \operQuadraVari{M^{c}}(dr,du) + \sum_{0<s \leq t} \, \Delta I_{s} \otimes \Delta I_{s}. 
\end{equation}
\item \label{RiemannRepreOperaQuadrVaria} For each $n \in \N$, let $(\tau^{n}_{k}: k \in \N)$ be a sequence of stopping times satisfying
\begin{enumerate}
    \item \label{prop1RandomParti} $0=\tau^{n}_{0} < \tau^{n}_{1} < \cdots < \tau^{n}_{k} < \cdots$
    \item \label{prop2RandomParti} $ \forall n \in \N$,  $\tau^{n}_{k} \nearrow \infty$ as $k \rightarrow \infty $ $\Prob$-a.e.
    \item \label{prop3RandomParti} $\sup_{k \in \N} (\tau^{n}_{k+1} - \tau^{n}_k ) \rightarrow 0$ as $n \rightarrow \infty$ $\Prob$-a.e.
\end{enumerate}\medskip

If $F$ is a $\goth{Bil}(G,G)$-valued adapted c\`{a}gl\`{a}d process,  then for any given $ t \geq 0$ we have 
\begin{flalign*} 
&\sum_{k=0}^{\infty} F(\tau_k^n \wedge t) \left(I_{\tau^{n}_{k+1} \wedge t}-I_{\tau^{n}_{k} \wedge t}, I_{\tau^{n}_{k+1} \wedge t}-I_{\tau^{n}_{k} \wedge t} \right) \\
& \mathop{\longrightarrow}^{\Prob}_{n \to \infty} \int_{0}^{t} F(s) \, d \left[\!\left[ X \right]\!\right]_s  =\int_{0}^{t} \!\!\int_U \mbox{Tr}_{\Phi(s, u) Q_{M^{c}}^{1 / 2}} F(s)   \operQuadraVari{M^{c}}(ds,du) +\sum_{0<s \leq t} F ( s )\left(\Delta I_{s},\Delta I_s\right). 
\end{flalign*}
\end{enumerate}
\end{proposition}
\begin{proof}
We begin by proving  \eqref{eqOptionalQuadraticVariationIntegral} and \eqref{eqOperatorOptionalQuadraticVariationIntegral}. In fact, since $I$ is an $H$-valued square integrable martingale, by Theorems 20.5 and 22.8 in \cite{Metivier} we have
$$ \left[ I \right]_{t} =\quadraVari{ I }^{c}_t + \sum_{0<s \leq t} \norm{\Delta I_{s}}_{G}^{2}$$
and 
$$ \left[\!\left[ I \right]\!\right]_t   = \operQuadraVari{ I }_{t}^{c}  + \sum_{0<s \leq t} \, \Delta I_{s} \otimes \Delta I_{s}.  $$
This way, it suffices to show that 
$$ 
\quadraVari{ I }^{c}_t = \int_{0}^{t} \!\!\int_U  \norm{ \Phi(r,u)Q_{M^{c}}^{1/2}(r,u)}_{\mathcal{HS}(H,G)}^{2} \operQuadraVari{M^{c}}(dr,du)
$$
and
$$ \operQuadraVari{ I }_{t}^{c}   =  \int_{0}^{t} \!\!\int_U  \Phi(r,u) Q_{M^{c}}(r,u) \Phi(r,u)^{*}  \operQuadraVari{M^{c}}(dr,du). $$

These identities are direct consequences of Theorem \ref{theoQuadraVariaStochIntegral} and \eqref{continuousintegralpart}.

To prove \ref{RiemannRepreOperaQuadrVaria}, observe that by Corollary 26.12 in \cite{Metivier} we have for any given $ t \geq 0$ that
$$ \sum_{k=0}^{\infty} F(\tau_k^n \wedge t) \left(I_{\tau^{n}_{k+1} \wedge t}-I_{\tau^{n}_{k} \wedge t}, I_{\tau^{n}_{k+1} \wedge t}-I_{\tau^{n}_{k} \wedge t} \right) \mathop{\longrightarrow}^{\Prob}_{n \to \infty} \int_{0}^{t} F(s) \, d \left[\!\left[ X \right]\!\right]_s.$$
Now, by \eqref{eqOperatorQuadraticVariationIntegral} and \eqref{eqOperatorOptionalQuadraticVariationIntegral}we have 
\begin{eqnarray*}
\int_{0}^{t} F(s) \, d \left[\!\left[ X \right]\!\right]_s 
& = & \int_{0}^{t} \!\!\int_U \left\langle F(s),\Phi(r,u) Q_{M^{c}}(r,u) \Phi(r,u)^{*} \right\rangle  \operQuadraVari{M^{c}}(dr,du) \\
& + & \sum_{0<s \leq t} \, F(s)(\Delta I_{s},\Delta I_{s}),
\end{eqnarray*}
where in the integral on the right-hand side is interpreted as the trace duality between $F(s) \in \goth{Bil}(G,G) \simeq \mathcal{L}(G,G)$ and $\Phi(r,u) Q_{M^{c}}(r,u) \Phi(r,u)^{*} \in \mathcal{N}(G,G)$. Here $\mathcal{N}(G,G)$ stands for the space of the nuclear or trace class operators on $G$ (see Theorem 7.10.40 in \cite{BogachevSmolyanov}, p.335). Moreover, by the definition of the trace duality if $(h_{n}:n \in \N)$ is an orthonormal basis in $H$ we have
\begin{eqnarray*}
\left\langle F(s),\Phi(r,u) Q_{M^{c}}(r,u) \Phi(r,u)^{*} \right\rangle 
& = &  \sum_{n=1}^{\infty} F(s) \left(\Phi(s, u) Q_{M^{c}}^{1 / 2} h_{n}, \Phi(s, u) Q_{M^{c}}^{1 / 2} h_{n} \right) \\
& = & \mbox{Tr}_{\Phi(s, u) Q_{M^{c}}^{1 / 2}} F(s). 
\end{eqnarray*}
Therefore
$$
\int_{0}^{t} \!\!\int_U \left\langle F(s),\Phi Q_{M^{c}} \Phi^{*} \right\rangle  \operQuadraVari{M^{c}}(dr,du) 
= \int_{0}^{t} \!\!\int_U \mbox{Tr}_{\Phi Q_{M^{c}}^{1/2}} F(s)   \operQuadraVari{M^{c}}(ds,du),
$$
which proves \ref{RiemannRepreOperaQuadrVaria}. 
\end{proof}

\section{It\^{o} formula}\label{sectItoFormula}

In this section we prove an It\^{o} formula for It\^{o} processes defined with respect to a cylindrical martingale-valued measure. For convenience, we first prove a real-valued version of the formula and then we state and prove the corresponding Hilbert space-valued version.  We start with some definitions. 

\begin{definition}\label{defiItoProcess}
Let $\xi$ be a $\mathcal{F}_{0}$-measurable $G$-valued random variable. Let $\beta:\Omega \times \R_{+} \rightarrow \R$ be adapted, with c\`adl\`ag paths and locally of finite variation. Let  $\psi: \Omega \times \R_{+} \rightarrow G$ be strongly progressively  measurable with paths in $L^1_{\textup loc} (\mathbb{R}_+, \beta ; G)$ $\Prob$-a.e. and $\Phi \in \Lambda^{2, \textup loc}(T, M)$. The $G$-valued adapted process 
\begin{equation}\label{eqDefiItoProcess}
 X_t : = \xi + \int_0^t \psi(s) d\beta(s) + \int_0^t \!\! \int_U \Phi(s, u)\, M(ds, du) 
\end{equation}
is called a (generalized) \emph{$G$-valued It\^{o} process}. 
\end{definition}

Consider $g\in G$ , $g^* \in G^{*}$ and $F \in \mathcal{L}(H,G)$.  Borrowing notation from \cite{Lu_Zhang} we consider the mappings $\operQuadraVari{F,g}_{G}$ and $\operQuadraVari{g^{*},F}_{G}:H \rightarrow \R$ given by
$$\operQuadraVari{F,g}_{G}(h)= \inner{F(h)}{g}_{G}, \quad  \operQuadraVari{g^{*},F}_{G}(h)= \left\langle g^{*}, F(h) \right\rangle_{G^{*},G}, \quad \forall h \in H.$$

The following is the real-valued version of our It\^{o} formula.

\begin{theorem}[It\^{o} formula, scalar case]\label{theoItoFormulaScalarCase}
Let $\xi$, $\beta$, $\psi$, $\Phi$ and $X$ be as in Definition \ref{defiItoProcess}. Let $f:[0, \infty) \times G \rightarrow \R$ be a function of class $C^{1,2}$. We assume $f_{xx}$ is uniformly continuous on bounded subsets of $[0,\infty)\times G$. Then, almost surely we have for every $t\geq 0$
$$
\begin{aligned}
& f(t, X_t) = f(0, \xi)+ \int_{0}^{t} f_{t}(s, X_{s-}) d s + \int_{0}^{t}\!\!\int_U \left\langle \! \left\langle f_{x}(s, X_{s-}), \Phi(s, u)\right\rangle \! \right\rangle_{G} M(ds, du) \\
& + \int_{0}^{t}\left\langle f_{x}(s, X_{s-}), \psi(s)\right\rangle_{G^{*}, G} d \beta(s)  +\frac{1}{2} \int_{0}^{t} \!\!\int_U \textup{Tr}_{\Phi(s, u) Q_{M^{c}}^{1 / 2}} \left( f_{x x}(s, X_{s-}) \right) \operQuadraVari{M^{c}}(ds,du) \\
& +\sum_{0<s \leq t}\left[f\left(s,X_s\right)-f\left(s,X_{s-}\right)-f_x\left(s,X_{s-}\right)\left(\Delta X_{s} \right)\right]
\end{aligned}
$$
\end{theorem}

\begin{remark}\label{remaItoFormulaEscalarCase}
The second integral in this It\^{o} formula is a stochastic integral with respect to $M$. In fact, observe that for each 
$(\omega, s, u)$ and $h \in H$ we have $f_{x}(s, X_{s-}) \in \mathcal{L}(G,\R) \simeq G^{*}$, therefore 
$$\left\langle\! \left\langle f_{x}(s, X_{s-}), \Phi(s, u)\right\rangle\!\right\rangle_{G}(h)= \left\langle f_{x}(s, X_{s-}), \Phi(s, u)h \right\rangle_{G^{*},G}.$$
Moreover, for every $h \in H$ the mapping 
\begin{multline*}
(\omega,s,u)  \mapsto  \left\langle \!\left\langle f_{x}(s, X_{s-}), \Phi(s, u)\right\rangle \! \right\rangle_{G}(Q_{M}^{1/2}(\omega,s,u)h) \\
=  \left\langle f_{x}(s, X_{s-}), \Phi(\omega, s, u)Q_{M}^{1/2}((\omega,s,u)h \right\rangle_{G^{*},G},    
\end{multline*}
is $\mathcal{P}_{T}\otimes \mathcal{B}(U)/\mathcal{B}(\R)$-measurable.

Furthermore, we have $\Prob$-a.e. 
\begin{multline*}
\int_{0}^{T}\!\! \int_{U} \norm{\left\langle \! \left\langle f_{x}(s, X_{s-}), \Phi(s, u)\right\rangle \! \right\rangle_{G} Q_{M}^{1/2} }^{2}_{\mathcal{HS}(H,\R)} \operQuadraVari{M}(ds,du) \\
\leq \sup_{s \in [0,T]} \norm{ f_{x}(s, X_{s-})}_{\mathcal{L}(G,\R)}^{2} 
\int_{0}^{T}\!\! \int_{U} \norm{\Phi(s,u) Q_{M}^{1/2} }^{2}_{\mathcal{HS}(H,G)} \operQuadraVari{M}(ds,du)<\infty. 
\end{multline*}
The remaining three integrals in the formula are defined almost surely as Lebesgue integrals with respect to $\beta$, the Lebesgue measure, and the quadratic variation $\operQuadraVari{M^{c}}$, respectively. The fact that these integrals are well defined can be checked by following similar arguments to those used above. For the last term, observe that since $X$ has c\`{a}dl\`{a}g paths for any given $\omega$, only a countable number of values of $s$ contribute to the sum. 
\end{remark}

Before we prove Theorem \ref{theoItoFormulaScalarCase} we need a couple of preliminary results. The first result establishes a regularity property for the remainder term in the second-order 
Taylor's formula. This result is a Banach space version of Lemma $4.93$ in \cite{KarandikarRao:2018}.

\begin{lemma}
\label{lemmaRemainderTaylorFormula}
Consider $f\in {\mathcal{C}}^{1,2}([0,\infty)\times V)$, where $V$ is an open convex set on a Banach space $E$. We assume $f_{xx}$ is uniformly continuous on bounded subsets of $[0,\infty)\times V$. For $(t,y,x)\in I\times V\times V$ define
$$
h(t,y,x) := f(t,y) - f(t,x) - f_x(t,x)(y-x) -\tfrac{1}{2} f_{xx}(t,x)(y-x,y-x).
$$
This function can be written as
\begin{equation}
\label{eqFormulaFor-h}
   h(t,y,x) = \int_{0}^{1} (1-s) \left[ f_{xx}(t,x+s(y-x))-f_{xx}(t,x) \right] (y-x,y-x)ds. 
\end{equation}

Besides, for $T<\infty$, $K\subseteq V$ and $\delta>0$ define
$$
\Gamma(T,K,\delta) = \sup\left\{ \frac{|h(t,x,y)|}{\|y-x\|^2}: 0 \leq t\leq T,x,y\in K, 0<\|y-x\| \leq \delta \right\},
$$
and
$$
\Lambda(T,K) = \sup\left\{ \frac{|h(t,x,y)|}{\|y-x\|^2}: 0 \leq t\leq T,x,y\in K, x\neq y \right\}.
$$
For fixed $T<\infty$ and $K$ bounded we have
$$
\Lambda(T,K)<\infty,\qquad \lim_{\delta\rightarrow 0}\Gamma(T,K,\delta) =0.
$$
\end{lemma}

\begin{proof}
Fixing $t\in I$, we apply Taylor's expansion with integral remainder to get 
$$
\label{TVMIntegral}
f(t,y) - f(t,x) - f_x(t,x)(y-x) =  \int_0^1 (1-s)f_{xx}(t,x+s(y-x))(y-x,y-x)ds.
$$
From this we readily get \eqref{eqFormulaFor-h}  and
$$
\| h(t,y,x) \| \leq \| y-x\|^2 \int_{0}^{1} (1-s) \| f_{xx}(t,x+s(y-x))-f_{xx}(t,x) \|_{\goth{Bil}(E,E)} ds.
$$
Let $\hat{K}$ be the convex closure of $K$, which is bounded and contained in $V$. By the uniform continuity of $f_{xx}$ on $[0,T]\times \hat{K}$, we get the result.
\end{proof}

Our second preliminary result is a Riemann representation formula for the stochastic integral with respect to the It\^{o} process $X$. In its proof, the reader may note the crucial role played by the associativity of the stochastic integral (Theorem \ref{theoAssociatiStochIntegral}). Indeed, if $M$ had an optional quadratic variation (as mentioned in the introduction for section \ref{sectQuadraVariaStochIntegral}), the result would follow from the classical theory as in \cite{Metivier}.

\begin{lemma}\label{lemmaRiemannRepresentation}
For each $n \in \N$, let $(\tau^{n}_{k}: k \in \N)$ be a sequence of stopping times satisfying
\begin{enumerate}
    \item $0=\tau^{n}_{0} < \tau^{n}_{1} < \cdots < \tau^{n}_{k} < \cdots$,
    \item $ \forall n \in \N$,  $\tau^{n}_{k} \uparrow \infty$ as $k \rightarrow \infty $ $\Prob$-a.e.,
    \item $\sup_{k \in \N} (\tau^{n}_{k+1} - \tau^{n}_k ) \rightarrow 0$ as $n \rightarrow \infty$ $\Prob$-a.e.
\end{enumerate}
Let $X$ be a $G$-valued It\^{o} process of the form \eqref{eqDefiItoProcess} and let $F$ be a $\mathcal{L}(G,\R)$-valued adapted c\`{a}dl\`{a}g process. 
Then, for fixed $t \geq 0$ we have
\begin{multline*}
  \sum_{k=0}^{\infty} F\left(\tau_k^n \wedge t \right)\left(X_{\tau^{n}_{k+1} \wedge t}-X_{\tau^{n}_{k} \wedge t}\right)  \\
\mathop{\longrightarrow}^{\Prob}_{n \to \infty} \int_{0}^{t}\left\langle F(s-), \psi(s)\right\rangle_{G^{*}, G} d \beta(s)
+ \int_{0}^{t}\!\!\int_U \left\langle \! \left\langle F(s-), \Phi(s, u)\right\rangle \! \right\rangle_{G} M(ds, du).
\end{multline*}
\end{lemma}
\begin{proof} Observe first that the process $(F(t-): t \geq 0)$ is predictable and locally bounded, therefore it is stochastically integrable with respect to $X$. 

Now, by the Riemann representation formula for the stochastic integral of a Hilbert space-valued semimartingale (see e.g. \cite{Metivier}, Section 26) we have
$$
\sum_{k=0}^{\infty} F\left(\tau_k^n \wedge t \right)\left(X_{\tau^{n}_{k+1} \wedge t}-X_{\tau^{n}_{k} \wedge t}\right)  
\mathop{\longrightarrow}^{\Prob}_{n \to \infty} \int_{0}^{t} F(s-) d X_s  .
$$
Consider the $G$-valued processes $Y$ and $Z$ given by
$$ Y(t)=\int_0^t \psi(s) d\beta(s), \quad  Z(t)=\int_0^t \!\! \int_U \Phi(s, u)\, M(ds, du). $$
Since $X_t=\xi+Y(t)+Z(t)$, by the linearity of the stochastic integral on the integrators we have 
$$
\int_{0}^{t}  F(s-)d X_s   =  \int_{0}^{t}  F(s-) d Y(s) + \int_{0}^{t} F(s-) d Z(s).     $$
Observe however that for a fixed $\omega \in \Omega$ we have
$$  \int_{0}^{t} F(s-)(\omega) d Y(s)(\omega) =  \int_{0}^{t}\left\langle F(s-)(\omega), \psi(s)\right\rangle_{G^{*}, G} d \beta(s)(\omega). $$
Furthermore, by Theorem \ref{theoAssociatiStochIntegral} we have
$$ \int_{0}^{t} F(s-) d Z(s) = \int_{0}^{t}\!\!\int_U \left\langle \! \left\langle F(s-), \Phi(s, u)\right\rangle \! \right\rangle_{G} M(ds, du).   $$
This completes the proof. 
\end{proof}

\begin{proof}[Proof of Theorem \ref{theoItoFormulaScalarCase}] For the reader's convenience, the proof is divided in two steps. We benefit from ideas used in the proofs of Theorem 2.139  in \cite{Lu_Zhang} and of Theorem 4.95 in \cite{KarandikarRao:2018}. \medskip

\textbf{Step 1.} We first prove the general case by assuming that the formula is valid under the assumption that the processes 
$$
\|X_{\cdot}\|_G,\quad \int_0^{\cdot}\|\psi(s)\|_G \, d \beta(s),\quad
\left\|\int_0^{\cdot} \!\!\int_U \Phi(s,u) M(ds,du)\right\|_G
$$
and
$$
\int_0^{\cdot} \!\!\int_U\left\|\Phi (s,u) Q_{M^{c}}^{1/2}\right\|_{\mathcal{HS}(H,G)}^2 \operQuadraVari{M^{c}}(ds,du)
$$
are uniformly bounded in $[0, t] \times \Omega
$. Define
\begin{eqnarray*}
   \tau_n & = & \inf \left\{s \in(0, t]: \left\|\int_0^s \!\!\int_U \Phi(r,u) M(dr,du)\right\|_G \right.\\
& {} & \hspace{25pt} \left. \lor \int_0^s \!\!\int_U\left\|\Phi(r,u) Q_{M^{c}}^{1/2} \right\|_{\mathcal{HS}(H,G)}^2 \operQuadraVari{M^{c}}(ds,du) \right.\left. \lor \int_0^s\|\psi(r)\|_G d \beta(r) \geqslant n\right\}
\end{eqnarray*}

(we take the infimum of the empty set as $t$). Observe that in the definition of $\tau_n$ the term $\norm{X_{\cdot}}_G$ is not present as its boundedness in guaranteed by that of the other terms. The general case can be deduced by using $\tau_n$ and considering the process
$$
X^n_t=\xi^{n}+\int_0^t \psi^n(s) d \beta(s)+\int_0^t \!\!\int_U \Phi^n(s, u)  M(ds,du),
$$ 
where $\xi^n(\omega)=\xi(\omega) \caract_{\{\|\xi(\omega)\|_G \leq n \}}$ and
$$
\psi^n(\omega,s)=\psi(\omega,s) \caract_{(0, \tau_n(\omega)]}(s), \qquad
\Phi^n(\omega,s, u)=\Phi(\omega,s,u) \caract_{(0, \tau_{n}(\omega)]}(s).
$$
Our assumption implies that the formula holds for $X^{n}$, that is 
$$
\begin{aligned}
& f(t, X^{n}_{t}) = f(0, \xi^{n})+\int_{0}^{t} f_{t}(s, X^n_{s-}) d s +\int_{0}^{t}\!\!\int_U \left\langle \! \left\langle f_{x}(s, X^n_{s-}), \Phi^{n}(s, u)\right\rangle \! \right\rangle_{G} M(ds, du) \\
& + \int_{0}^{t}\left\langle f_{x}(s, X^n_{s-}), \psi^{n}(s)\right\rangle_{G^{*}, G} d \beta(s) +\frac{1}{2} \int_{0}^{t} \!\!\int_U \mbox{Tr}_{\Phi^{n}(s, u) Q_{M^{c}}^{1 / 2}} \left( f_{x x}(s, X^n_{s-}) \right)  \operQuadraVari{M^{c}}(ds,du) \\
& +\sum_{0<s \leq t}\left[f\left(s,X^n_s\right)-f\left(s,X^n_{s-}\right)-f_x\left(s,X^n_{s-}\right)\left(\Delta X^n_{s} \right)\right].
\end{aligned}
$$

We must check that both sides in the above equality converge (at least for some subsequence) almost surely to the corresponding sides in the It\^{o} formula given by the Theorem. \medskip

Our first task is to check that  $f\left(t,X^n_t\right) \stackrel{ucp}{\longrightarrow} f(t, X_t)$. In order to do so,
we start by showing that $X^n \stackrel{ucp}{\longrightarrow} X$. In fact, by dominated convergence we have $\omega$-wise  
$$
\int_0^t \psi^{n}(s) d\beta(s) \rightarrow \int_0^t \psi(s) d\beta(s).
$$
Similarly, by monotone convergence we have 
\begin{flalign*}
& \Exp \int_0^T \!\!\int_U \norm{ \left(\Phi^{n}(s,u)-\Phi(s,u) \right)Q_{M}^{1/2}}_{\mathcal{HS}(H,G)} \operQuadraVari{M}(ds,du)  \\
&=\Exp \int_0^T \!\!\int_U \caract_{(\tau_n,t]}\norm{\Phi(s,u)Q_{M}^{1/2}}_{\mathcal{HS}(H,G)}  \operQuadraVari{M}(ds,du)  \searrow 0.     
\end{flalign*}

This shows that $\Phi^n \rightarrow \Phi$ in $\Lambda^2(M, T)$. By Corollary 6.20 in \cite{CCFM:SPDE} we have
$$
\int_0^{\cdot} \!\!\int_U \Phi^n(s,u) M(ds,du) \stackrel{ucp}{\longrightarrow} \int_0^{\cdot} \!\!\int_U\Phi(s,u) M(ds,du).
$$

Therefore $X^{n} \stackrel{ucp}{\longrightarrow}   X$ and so  $f\left(t, X^n_t\right) \stackrel{ucp}{\longrightarrow}  f(t, X_t)$. We also have  $\omega$-wise, $f\left(t, \xi^n \right) \rightarrow  f(t, \xi)$. 

Next we shall prove
\begin{multline}\label{eqUCPConverStochIntegItoFormulaApproxi}
\int_0^t \!\!\int_U \langle \! \langle f_{x} (s, X^n_{s-}), \Phi^n(s, u)\rangle \! \rangle_G \, M(d s, d u) \\
 \stackrel{ucp}{\longrightarrow}  \int_0^t \!\!\int_U \left\langle \! \left\langle f_{x}(s, X_{s-}), \Phi(s, u)\right\rangle \! \right\rangle_G \, M(d s, d u). 
\end{multline}
To do this, observe that the uniform continuity of $f_x$ and the definition of $\tau_n$ guarantee the existence of a constant $C>0$ such that
$\mathbbm{1}_{\left(\tau_n, t\right]} \| f_x (s, X_{s-}) \|_G^2 \leq C^2$. This implies  
$$
\begin{aligned}
& \mathbb{E} \int_0^t \int_{U} \norm{ \left(\langle \! \langle f_{x} (s, X^n_{s-}), \Phi^n(s, u)\rangle \! \rangle_G
- \left\langle \! \left\langle f_{x}(s, X_{s-}), \Phi(s, u)  \right\rangle \! \right\rangle_G \right) Q^{1/2}_{M} }_{\mathcal{HS}(G, \R)}^2 \operQuadraVari{M}(ds,du)   \\
& =\mathbb{E} \int_0^t \int_{U} \mathbbm{1}_{\left(\tau_n, t\right]} \norm{
\left\langle \! \left\langle f_{x}(s, X_{s-}), \Phi(s, u) \right\rangle \! \right\rangle_G Q^{1/2}_{M} }_{\mathcal{HS}(G, \R)}^2  \operQuadraVari{M}(ds,du)  \\
& \leq \mathbb{E} \int_0^t \int_{U} \mathbbm{1}_{\left(\tau_n, t\right]} \norm{
f_{x}(s, X_{s-})}^{2}_{G} \norm{\Phi(s, u) Q^{1/2}_{M}}_{\mathcal{HS}(H,G)}^2  \operQuadraVari{M}(ds,du)  \\
&  \leq C^2\|\Phi\|_{\Lambda^2(M, T)}^2<\infty.
\end{aligned}
$$

This allows us to use dominated convergence to deduce that
$$
 \mathbb{E} \int_0^t \int_{U} \norm{ \left(\langle \! \langle f_{x} (s, X^n_{s-}), \Phi^n(s, u)\rangle \! \rangle_G
- \left\langle \! \left\langle f_{x}(s, X_{s-}), \Phi(s, u)  \right\rangle \! \right\rangle_G \right) Q^{1/2}_{M} }_{\mathcal{HS}(G, \R)}^2 \operQuadraVari{M}(ds,du)
$$
converges to $0$. This and Corollary 6.20 in \cite{CCFM:SPDE} give us \eqref{eqUCPConverStochIntegItoFormulaApproxi}. \medskip

Following similar arguments we have $\Prob$-a.e. 
$$ 
\int_{0}^{t} f_{t}(s, X^n_{s-}) d s \longrightarrow \int_{0}^{t} f_{t}(s, X_{s-}) d s, 
$$
$$
\int_{0}^{t}\left\langle f_{x}(s, X^n_{s-}), \psi^{n}(s)\right\rangle_{G^{*}, G} d \beta(s) \longrightarrow \int_{0}^{t}\left\langle f_{x}(s, X_{s-}), \psi(s) \right\rangle_{G^{*}, G} d \beta(s),  
$$
and
\begin{multline*}
\int_{0}^{t} \!\!\int_U \mbox{Tr}_{\Phi^{n}(s, u) Q_{M^{c}}^{1 / 2}} \left( f_{x x}(s, X^n_{s-}) \right)
 \operQuadraVari{M^{c}}(ds,du) \\
 \longrightarrow 
 \int_{0}^{t} \!\!\int_U \mbox{Tr}_{\Phi(s, u) Q_{M^{c}}^{1 / 2}} \left( f_{x x}(s, X_{s-}) \right) \operQuadraVari{M^{c}}(ds,du).
\end{multline*}

It remains to deal with the convergence of the discrete part. It is enough to prove that
\begin{equation}
\label{labelConvNecesaria}
\sum_{0<s\leq t} \tilde{h}(s,X^n_s,X^n_{s-}) \rightarrow \sum_{0<s\leq t} \tilde{h}(s,X_s,X_{s-}) \quad \Prob-a.e.
\end{equation}
where $\tilde{h}(t,y,x)$ is defined as 
$$
\tilde{h}(t,y,x) := f(t,y) - f(t,x) - f_x(t,x)(y-x). 
$$

For any given $t \geq 0$, $\omega \in \Omega$, let 
$$K^{t,\omega}= \{ X_{s}(\omega): 0 \leq s \leq t\} \cup \{ X_{s-}(\omega): 0 \leq s \leq t \}.$$
Since $X$ have c\`adl\`ag paths, according to Proposition $1$ in \cite{Jakubowski:1986} the set $K^{t,\omega}$ is compact and coincides with the closure of the set 
$$ 
\{ X_{s}(\omega): 0 \leq s \leq t\}.
$$

Since $X^n_s \stackrel{ucp}{\longrightarrow} X_s$, $X^n_{s-} \stackrel{ucp}{\longrightarrow} X_{s-}$, and $\tilde{h}(s,\cdot,\cdot)$ is continuous, we have 
$$
\tilde{h}(s,X^n_s,X^n_{s-}) \stackrel{ucp}{\longrightarrow} \tilde{h}(s,X_s,X_{s-})\quad \textup{as } n\rightarrow \infty
$$
so we can assume (passing to a subsequence if necessary) that convergence is $\Prob$-a.e. Moreover, observe that 
$$ \tilde{h}(t,y,x) = h(t,y,x) + \tfrac{1}{2} f_{x x}(t,x)(y-x,y-x), $$
where $h$ is the function defined in Lemma \ref{lemmaRemainderTaylorFormula}. Uniform continuity of $f_{xx}$ and the definition of $\tau_n$ give the existence of $C>0$ such that
$\mathbbm{1}_{\left(\tau_n, t\right]} \| f_{xx} (s, X_{s-}) \|_{\goth{Bil}(G,G)} \leq C$. For fixed $\omega$, by Lemma \ref{lemmaRemainderTaylorFormula} we have
\begin{eqnarray*}
   \sum_{0<s\leq t} \sup_{n\in \N} |\tilde{h}(s,X^n_s,X^n_{s-})| & \leq & \left(\Lambda(t,K^{t,\omega}) +\frac{C}{2} \right)\sum_{0<s\leq t} \norm{\Delta X_s}_G^2 \\
%   & \leq & \Lambda(K^{t,\omega}) \sum_{0<s\leq t} \norm{\Delta X_s}_G^2\\
   & \leq & \left(\Lambda(t,K^{t,\omega}) +\frac{C}{2} \right) [X]_t < \infty \quad \Prob-a.e.
\end{eqnarray*}
By Wieierstrass M-test we obtain uniform convergence in $n$ for  the sums on the left-hand side of \eqref{labelConvNecesaria}. This allows us to interchange the limit on $n$ with the sum over $s$, obtaining \eqref{labelConvNecesaria}.

Since convergence in probability implies the existence of a subsequence converging almost surely, we have shown the general case of our  It\^{o} formula under the assumption that it holds for the uniformly bounded case. \bigskip

\textbf{Step 2.} According to Step 1, we assume that $\|X_{\cdot}\|_{G}$, $\int_0^{\cdot} \| \psi(s)\|_G \, d\beta(s)$, as well as
$$
\left\| \int_{0}^{\cdot} \!\!\int_U \Phi(s, u) M(ds,du) \right\|_{G}, \quad
\int_0^{\cdot} \!\!\int_U \left\|  \Phi(s,u)Q_{M^{c}}^{1/2}\right\|^2_{\mathcal{HS}(H,G)}  \operQuadraVari{M^{c}}(ds,du)
$$
are uniformly bounded on $[0, T] \times \Omega$ by a constant $K$. \medskip

For each $n \in \N$, define a sequence $(\tau^{n}_{k}: k \in \N)$ of stopping times inductively as follows: $\tau^{n}_{0}=0$ and for $k \in \N$, 

$$\tau^n_{k+1} = \tilde{\tau}^n_{k+1} \land (\tau^n_k + 2^{-n}) $$ 
where
$$
\tilde{\tau}^n_{k+1}=\inf \left\{s > \tau^{n}_k : \max \left\{Z^n_k(s), Z^n_k(s-)\right\} \geq 2^{-n}\right\}
$$ 
with 
\begin{multline*}
Z^n_k(s) = \left\|\int_{\tau^n_{k}}^{s} \!\!\int_U \Phi(r, u) M(dr,du)\right\|_{G}  \\
\lor \int_{\tau^n_{k}}^{s} \!\!\int_U\left\|\Phi(r, u) Q_{M^{c}}^{1 /2}\right\|_{\mathcal{HS}(H,G)}^{2}  \operQuadraVari{M^{c}}(ds,du) 
\lor \int_{\tau^n_{k}}^{s} \| \psi(s)\|_{G} \, d \beta(s).
\end{multline*}
Notice that 
$$0=\tau^{n}_{0} < \tau^{n}_{1} < \cdots < \tau^{n}_{k} < \cdots ,$$ 
$$ \forall n \in \N, \quad \tau^{n}_{k} \uparrow \infty \mbox{ as } k \rightarrow \infty, $$
$$ (\tau^{n}_{k+1} - \tau^{n}_k ) \leq 2^{-n},$$   
and
$$
 \left\|\int_{\tau^{n}_{k}}^{\tau^{n}_{k+1}} \!\!\int_U \Phi(s,u) M(ds,du)\right\|_{G} \leq 2^{-n},\quad  \int_{\tau^{n}_{k}}^{\tau^{n}_{k+1}} \| \psi(s)\|_{G} d\beta(s) \leq 2^{-n}.
 $$

Fix $0 \leq t \leq T$. Observe
$$ 
I = f(t,X_t) - f(0,\xi) = \sum_{k=0}^{\infty}\left(U_k^n+V_k^n+ W_k^n +R_k^n\right) 
$$
where
\begin{eqnarray*}
U_k^n & = & f\left(\tau_{k+1}^n \wedge t, X_{\tau_{k+1}^n \wedge t} \right)-f\left(\tau_k^n \wedge t, X _{\tau_{k+1}^n \wedge t} \right) \\  
V_k^n & = & f_x\left(\tau_k^n \wedge t, X _{\tau^{n}_{k} \wedge t} \right)\left(X_{\tau^{n}_{k+1} \wedge t}-X_{\tau^{n}_{k} \wedge t}\right) \\ 
W_k^n & = & \tfrac{1}{2} f_{x x}\left(\tau_k^n \wedge t, X_{\tau_{k}^n \wedge t}  \right) \left(X_{\tau_{k+1}^n \wedge t}  - X_{\tau^{n}_{k} \wedge t}, X_{\tau_{k+1}^n \wedge t}  - X_{\tau^{n}_{k} \wedge t} \right) \\ 
R_k^n & = & h\left(\tau_k^n \wedge t, X_{\tau^n_{k+1} \wedge t}, X_{\tau^n_{k}\wedge t} \right)
\end{eqnarray*}

Define
$$ A_n=\sum_{k=0}^{\infty} U_k^n, \quad B_n=\sum_{k=0}^{\infty} V_k^n, \quad  C_n=\sum_{k=0}^{\infty} W_k^n  
\quad
D_n=\sum_{k=0}^{\infty} R_k^n.
$$ 

Notice that $\Prob$-a.e.
$$
A_{n} = \sum_{k=0}^{\infty} \int_{\tau^{n}_{k} \wedge t}^{\tau^{n}_{k+1} \wedge t} f_{s}(s,X_{\tau^{n}_{k+1} \wedge t}) \, ds  \rightarrow \int_{0}^{t} f_{s}(s, X_{s}) ds,
$$
when $n \rightarrow \infty$. Since for $\Prob$-a.e. $\omega$ we have $X_s(\omega)=X_{s-}(\omega)$ for all but countably many values of $s$, we conclude that 
$$A_{n} \mathop{\longrightarrow}^{\Prob}_{n \to \infty} \int_{0}^{t} f_{s}(s, X_{s-}) ds.
$$
Now, if we apply Lemma \ref{lemmaRiemannRepresentation} to $F(s)=f_{x}(s, X_s)$ we have 

$$
B_{n} \mathop{\longrightarrow}^{\Prob}_{n \to \infty} \int_{0}^{t}\langle f_{x}(s, X_{s-}, \psi(s)\rangle_{G^*, G} \, d\beta(s) + \int_{0}^{t} \!\!\int_U \operQuadraVari{f_{x}(s,X_{s-}),\Phi(s, u) }_{G} M(ds,du).
$$

By Proposition \ref{propRiemannQuadraticVariation} applied to $F(s)=f_{x x}(s, X_{s-})$ we conclude that 
$$ 
C_n \mathop{\longrightarrow}^{\Prob}_{n \to \infty} \frac{1}{2} \int_{0}^{t} \!\!\int_U \mbox{Tr}_{\Phi(s, u) Q_{M^{c}}^{1 / 2}} f_{x x}(s, X_{s-})   \operQuadraVari{M^{c}}(ds,du) + \frac{1}{2} \sum_{0<s \leq t} f_{x x}(s, X_{s-})\left(\Delta I_{s},\Delta I_s\right).
$$

To complete the proof, it suffices to prove that
$$
D_n = \sum_{k=0}^{\infty} \, h(\tau_k^n \wedge t, X_{\tau_{k+1}^n \wedge t} , X_{\tau_{k}^n \wedge t} ) \xrightarrow{\Prob} \sum_{0<s \leq t} \, h\left(s, X_s, X_{s-}\right).
$$ 
Now we partition the index $k$ into three sets: 
$$
H^n(\omega)=\left\{k \geqslant 0: \tau_{k+1}^n(\omega) \wedge t=\tau_k^n(\omega) \wedge t\right\}.
$$

$$
E^n(\omega)=\left\{k \notin H^n:\left\|\Delta X_{\tau_{k+1}^n \wedge t}(\omega)\right\|_G \leq  4 \cdot 2^{-n}\right\}.
$$

$$
F^n(\omega)=\left\{k \notin H^n:\left\|\Delta X_{\tau_{k+1}^n \wedge t}(\omega)\right\|_G > 4 \cdot 2^{-n}\right\}.
$$

Observe that the sets $H^{n}(\omega)$, $E^{n}(\omega)$ and $F^{n}(\omega)$ correspond respectively to the cases where $h (\tau_k^n \wedge t, X_{\tau_{k+1}^n \wedge t} , X_{\tau_{k}^n \wedge t} )$ is zero, small, and large.

For $k \in H^n$,  $h\left(\tau_k^n(\omega) \wedge t, X_{\tau_{k+1}^n}(\omega), X_{\tau_{k+1}^n}(\omega)\right)=0$. 
Therefore $D_n = D_n^1 + D_n^2$, where 
$$
\begin{aligned} 
& D_n^1=\sum_{k \in E^n(\omega)} h\left(\tau_k^n(\omega) \wedge t, X_{\tau_{k+1}^n \wedge t}(\omega), X_{\tau_{k}^n \wedge t}(\omega)\right) \\ 
& D_n^2=\sum_{k \in F^{n}(\omega)}  h\left(\tau_k^n(\omega) \wedge t, X_{\tau_{k+1}^n \wedge t}(\omega), X_{\tau_k^n \wedge t}(\omega)\right) .
\end{aligned} 
$$

Given $k$, we take $v,w\in (\tau^n_k, \tau^n_{k+1} )$. We have

\begin{eqnarray*}
\left\|X_v-X_w\right\|_G & \leq & 
\left\|X_v-X_{\tau_{k}^n} \right\|_G + \left\|X_w-X_{\tau_{k}^n} \right\|_G \\
 & \leq & 2 \cdot \max_{\tau_{k}^n < v < \tau_{k+1}^n} \left\{ \left\| \int_{\tau_{k}^n}^{v} \int_U \Phi(r,u) M(dr,du) \right\|_G + \int_{\tau_k^n}^{v} \norm{\psi(s)}_G d\beta(s) \right\} \\
& \leq & 4 \cdot 2^{-n}.
\end{eqnarray*}

It follows that $v \in (\tau_k^n, \tau_{k+1}^n)$ implies $\abs{X_{v^-} - X_v} \leq 4 \cdot 2^{-n}$, so for $s\in (0,t]$
\begin{equation}
\label{eqSizeOfJumpsItoFormula}
%\text{\color{red} *}
\left\|\Delta X_s(\omega)\right\|_G > 4 \cdot 2^{-n} \Rightarrow
s =\tau_{k+1}^n \wedge t \,\, \text{ for some } k \in F^n(\omega).
\end{equation}

Therefore, for $k \in E^n(\omega)$
$$
\left\|X_{\tau_{k+1}^n}(\omega)-X_{\tau_k^n \wedge t}(\omega)\right\|_G \leq \left\|X_{(\tau_{k+1}^n \wedge t)^{-}}(\omega)-X_{\tau_k^n \wedge t}(\omega)\right\|_G
 +\left\|\Delta X_{\tau_{k+1}^n \wedge t}(\omega)\right\|_G \leq 6 \cdot 2^{-n}.
$$

Consequently
\begin{eqnarray*}
\left|D_n^1\right| & \leq & \sum_{k \in E^{n}(\omega)}\left| h \left(\tau_k^n(\omega) \wedge t, X_{\tau_{k+1}^n \wedge t}(\omega), X_{\tau_k^n \wedge t}(\omega)\right) \right| \\
 & \leq & \Gamma\left(K^{t,\omega}, 6 \cdot 2^{-n}\right) \sum_{k=0}^{\infty} \left\| X_{\tau_{k+1}^n \wedge t}(\omega)-X_{\tau_{k \wedge t}^n}(\omega)\right\|_G^2.
\end{eqnarray*}

Since 
$$
\sum_{k=0}^{\infty}\left\|X_{\tau_{k+1}^n \wedge t}-X_{\tau_{k}^n \wedge t} \right\|_G^2 \xrightarrow{\Prob}[X]_t
$$
and $\Gamma\left(K^{t,\omega}, 6 \cdot 2^{-n}\right) \rightarrow 0$ for fixed $\omega$, by \eqref{eqSizeOfJumpsItoFormula} it follows that 
$$
D_n^1 \xrightarrow{\Prob} 0.
$$

Now, we must show that
$$
D_n^2 \xrightarrow{\Prob} \sum_{0<s \leq t}   h\left(s, X_s, X_{s^-}  \right).
$$

Let $J(\omega)=\left\{0 < s \leq t:\norm{\Delta X_s(\omega)}_{G} >0\right\}$.
Since $X$ is c\`{a}dl\`{a}g and $G$ is Hilbert, $J(\omega)$ has a countable number of elements for $\Prob$-a.e. $\omega$.

Given $\omega\in\Omega$ and $s\in J(\omega)$, define
$$
a_s^n(\omega):=\sum_{k \in F^n(\omega)}  
 h\left(\tau_k^n(\omega) \wedge t, X_{\tau_{k+1}^n  \wedge t}(\omega) , X_{\tau_k^n \wedge t}(\omega)  \right) \caract_{ \{ (\tau_{k+1}^n(\omega)  \wedge t)= s \}}(\omega) .
$$
We then have
$$
D_n^2(\omega) = \sum_{s \in J(w)} a_s^n(\omega).
$$

%\section{Pizarra 6}

Because of \eqref{eqSizeOfJumpsItoFormula}, if $\| \Delta X_s(\omega) \|_{G} > 4 \cdot 2^{-n}$ we have

$$
a_s^n(w)=  h\left(\tau_k^n(\omega) \wedge t, X_s(\omega), X_{\tau_k^n \wedge t}(w)\right) \text { where } s = \tau_{k+1}^n(\omega) \wedge t.
$$

This shows that 
\begin{equation}
    \label{eqConv-ah}
    a_s^n(\omega) \rightarrow  h\left(s, X_s(w), X_{s^-}(\omega)\right),\quad \Prob-\text{a.e.}
\end{equation}

For $k \in F^n(\omega)$
\begin{eqnarray*}
\left\|X_{\tau_{k+1}^n \wedge t}(w)-X_{\tau_{k}^n  \wedge t}(\omega) \right\|_G  & \leq & \left\|X_{\left(\tau_{k+1}^n \wedge t\right)-}(\omega)-X_{\tau_k^n \wedge t}(w)\right\|_G + \left\|\Delta X_{\tau_{k+1}^n \wedge t}(\omega) \right\|_G \\
& \leq & 4 \cdot 2^{-n}+\left\|\Delta X_{\tau_{k+1}^n \wedge t} (\omega)\right\|_{G} \\
& \leq & 2 \left\|\Delta X_{\tau_{k+1}^n \wedge t} (\omega)\right\|_{G}.
\end{eqnarray*}

Given $\left|a_s^n(w)\right| \neq 0$, it follows that $s=\tau_{k+1}^n(\omega) \wedge t$ for some $k \in F^n(\omega)$, so
$$
\begin{aligned}
\left|a_s^n(\omega)\right| & \leq \Lambda\left(K^{t, \omega}\right)\left\|X_{\tau_{k+1}^n \wedge t}(\omega)-X_{\tau_k^n \wedge t}(\omega)\right\|_G^2 \\
& \leq 4 \Lambda(K^{t,\omega}) \left\|\Delta X_{\tau_{k+1}^n \wedge t}(\omega)\right\|_G^2 \\
& = 4 \Lambda\left(K^{t, \omega}\right)\left\|\Delta X_s(\omega)\right\|_G^2.
\end{aligned}
$$

In this way
$$
\sum_{s \in J(\omega)}\left|a_s^n(\omega)\right| \leq 4 \Lambda\left(K^{t,\omega}\right) \sum_{s \in J(\omega)}\left\|\Delta X_s(\omega)\right\|_G^2
$$
and we have
$$
\sum_{s \in J(\omega)}\left\|\Delta X_s(\omega)\right\|_G^2 \leq[X]_t(\omega)<\infty \quad \Prob\text{-a.e.}
$$
By Weierstrass $M$-test and \eqref{eqConv-ah} we conclude
$$
D_n^2=\sum_{s \in J(\omega)} a_s^n(\omega) \rightarrow \sum_{s \in J(\omega)} h\left(s, X_s(\omega), X_{s-}(\omega)\right) \quad \Prob\text{-a.e.} 
$$
and this implies convergence in probability. This completes the proof. 
\end{proof}

\begin{remark}
    It is worth mentioning some aspects on the arguments used in the proof of Theorem \ref{theoItoFormulaScalarCase}. In Step 1 we follow ideas in the proof of Theorem 2.139 in \cite{Lu_Zhang}, with careful modifications due to the presence of jumps in the It\^{o} process. In particular, the treatment of jumps was rather based on the proof of Theorem 4.95 of \cite{KarandikarRao:2018}. Arguments in Step 2 are based mainly in modifications of the proof of Theorem 4.95 of \cite{KarandikarRao:2018}, which are for the finite dimensional semimartingale case. Since we do not have a suitable definition of an integral for a cylindrical semimartingale-valued measure, we can not treat the It\^{o} process as a semimartingale. Therefore, a particular approach must be used for each type of integral involved in the definition of the It\^{o} process. Besides, as $M$ does not have an optional quadratic variation, we were forced to reformulate Riemann type approximations for the stochastic integral an its quadratic variation (Proposition \ref{propRiemannQuadraticVariation} and Lemma \ref{lemmaRiemannRepresentation}).
\end{remark}

Now we state and prove the Hilbert space-valued version of last theorem.

\begin{theorem}[It\^{o} formula, vector case]\label{theoItoFormulaVectorCase}
Let $\xi$, $\beta$, $\psi$, $\Phi$ and $X$ be as in Definition \ref{defiItoProcess}. Let $f:[0, \infty) \times G \rightarrow K$ be a function of class $C^{1,2}$. We assume $f_{xx}$ is uniformly continuous on bounded subsets of $[0,\infty)\times G$. Then  almost surely we have for every $t\geq 0$
$$
\begin{aligned}
& f(t, X_t) = f(0, \xi)+ \int_{0}^{t} f_{t}(s, X_{s-}) d s +\int_{0}^{t}\!\!\int_U  f_{x}(s, X_{s-}) \Phi(s, u)  M(ds, du) \\
& + \int_{0}^{t} f_{x}(s, X_{s-}) \psi(s) d \beta(s)   +\frac{1}{2} \int_{0}^{t} \!\!\int_U 
\mbox{Tr}_{\Phi(s, u) Q_{M^{c}}^{1 / 2}} \left( f_{x x}(s, X_{s-}) \right)
\operQuadraVari{M^{c}}(ds,du) \\
 & +\sum_{0<s \leq t}\left[f\left(s,X_s\right)-f\left(s,X_{s-}\right)-f_x\left(s,X_{s-}\right)\left(\Delta X_{s} \right)\right]
\end{aligned}
$$
\end{theorem}

\begin{remark}
The second integral in the first line is a stochastic integral. The stochastic integrability of $(\omega,s,u)  \mapsto  f_{x}(s, X_{s-}(\omega)) \Phi(\omega,s, u)$ with respect to $M$ can be proved by following similar arguments to those used in Remark \ref{remaItoFormulaEscalarCase}. The remaining three integrals are defined almost surely as $K$-valued Bochner integrals with respect to $\beta$, the Lebesgue measure, and the quadratic variation $\operQuadraVari{M^{c}}$. 
\end{remark}

\begin{proof}[Proof of Theorem \ref{theoItoFormulaVectorCase}]
Let $k \in K$ and consider the function $F:[0,T] \times G \rightarrow \R$ defined as $F(t,g)= \inner{f(t,g) }{ k}_{K}$. Then $F$ is of class $C^{1,2}$ and for every $h \in H$ we have
\begin{eqnarray*}
\left\langle \! \left\langle F_{x}(s, X_{s-}), \Phi(s, u)\right\rangle \! \right\rangle_{G}(h)
%& = &  \left\langle \! \left\langle \, \left\langle  f_{x}(s, X_s), k \right\rangle_{K},  \Phi(s, u) \right\rangle \! \right\rangle_{G}(h) \\
& = &  \left\langle \, ( f_{x}(s, X_{s-}), k )_{K}, \Phi(s, u)h \right\rangle_{G^{*},G} \\
& = & ( f_{x}(s, X_{s-})\Phi(s,u), k  )_{K}(h).    
\end{eqnarray*}
Therefore, almost surely
$$ 
\int_{0}^{t}\!\!\int_U \left\langle \! \left\langle F_{x}(s, X_{s-}), \Phi(s, u)\right\rangle \! \right\rangle_{G} M(ds, du)
= \inner{ \int_{0}^{t}\!\!\int_U  f_{x}(s, X_{s-}), \Phi(s, u)  M(ds, du)}{  k }_{K}.
$$
Likewise, one can show that almost surely 
$$\int_{0}^{t} F_{t}(s, X_{s-}) d s =   \inner{  \int_{0}^{t} f_{t}(s, X_{s-}) d s}{  k }_{K}
$$
$$ \int_{0}^{t}\left\langle F_{x}(s, X_{s-}), \psi(s)\right\rangle_{G^{*}, G} d \beta(s) = 
\inner{\int_{0}^{t} f_{x}(s, X_{s-}) \psi(s) d \beta(s)}{k}_{K}   $$
\begin{multline*}
\int_{0}^{t} \!\!\int_U \mbox{Tr}_{\Phi(s, u) Q_{M^{c}}^{1 / 2}} \left( F_{x x}(s, X_{s-}) \right) \operQuadraVari{M^{c}}(ds,du) \\
= \inner{\int_{0}^{t} \!\!\int_U 
\mbox{Tr}_{\Phi(s, u) Q_{M^{c}}^{1 / 2}} \left( f_{x x}(s, X_{s-}) \right) \operQuadraVari{M^{c}}(ds,du)}{k}_{K}   
\end{multline*}
and 
\begin{multline*}
\sum_{0<s \leq t}\left[F\left(s,X_s\right)-F\left(s,X_{s-}\right)-F_x\left(s,X_{s-}\right)\left(\Delta X_{s} \right)  \right] \\
= \inner{\sum_{0<s \leq t}\left[f\left(s,X_s\right)-f\left(s,X_{s-}\right)-f_x\left(s,X_{s-}\right)\left(\Delta X_{s} \right)  \right] }{k }_{K}.
\end{multline*}

Hence by applying the scalar case (Theorem \ref{theoItoFormulaScalarCase}) to $F$ and by using the above identities, almost surely we have 
\begin{eqnarray*}
( f(t, X_t), k \, )_{K} 
& = &  \biggl( \, f(0, \xi)+\int_{0}^{t}\!\!\int_U  f_{x}(s, X_{s-}) \Phi(s, u)  M(ds, du)   \\
&  {} & +\int_{0}^{t} f_{t}(s, X_{s-}) d s + \int_{0}^{t} f_{x}(s, X_{s-}) \psi(s) d \beta(s)    \\
& {} &  +\frac{1}{2} \int_{0}^{t} \!\!\int_U 
\mbox{Tr}_{\Phi(s, u) Q_{M^{c}}^{1 / 2}} \left( f_{x x}(s, X_{s-}) \right) \operQuadraVari{M^{c}}(ds,du)  \\
& {} & + \sum_{0<s \leq t}\left[f\left(s,X_s\right)-f\left(s,X_{s-}\right)-f_x\left(s,X_{s-}\right)\left(\Delta X_{s} \right)  \right]  ,k \,  \biggr)_{K}.
\end{eqnarray*}

% & {} & -\tfrac{1}{2} \sum_{0<s \leq t} \left[ f_{xx} \left(s,X_{s-}\right)\left(\Delta X_{s},\Delta X_s\right)  \right] 

Let $(k_{i}:i \in \N)$ be a dense subset in $K$ and let $L$ denote the $\mathcal{Q}$-span of $(k_{i}:i \in \N)$. Then almost surely the above identity holds true for every $k \in L$. By the Hahn-Banach theorem we obtain the result. 
\end{proof}

\begin{corollary}[It\^{o} formula in the continuous paths setting]\label{coroItoVectorContinuousCase}

With the notation in Theorem \ref{theoItoFormulaVectorCase}, assume further that for each $A\in\mathcal{A}$ and $h\in H$ the real-valued stochastic process $M(\cdot,A)(h)$ is continuous.  Then  almost surely we have for every $t\geq 0$
$$
\begin{aligned}
& f(t, X_t) = f(0, \xi)+\int_{0}^{t}\!\!\int_U  f_{x}(s, X_s) \Phi(s, u)  M(ds, du)  + \int_{0}^{t} f_{x}(s, X_s) \psi(s) d \beta(s)  \\
&  +\int_{0}^{t} f_{t}(s, X_s) d s +\frac{1}{2} \int_{0}^{t} \!\!\int_U 
\mbox{Tr}_{\Phi(s, u) Q_{M}^{1 / 2}} \left( f_{x x}(s, X_s) \right)
\operQuadraVari{M}(ds,du) 
\end{aligned}
$$
\end{corollary}
\begin{proof}
Our assumption on $M$ and Proposition 6.11 in \cite{CCFM:SPDE} shows that the stochastic integral $\int_{0}^{\cdot} \int_{U} \Phi (s,u) M(ds,du)  $ has continuous paths. Therefore $X$ has continuous paths as well.  Hence the last term in the formula of Theorem \ref{theoItoFormulaVectorCase} vanishes.     
\end{proof}

\section{Application to a Burkholder inequality}\label{subSectBurkholder} 

In this section we show how an implementation of our It\^{o} formula by writing a self contained proof of Burkholder type inequalities for stochastic integrals defined with respect to  a cylindrical martingale-valued measure. It is worth noting, that since the integrals involved are radonified, such results can be obtained from the corresponding Burkholder-Davis-Gundy inequalities for Hilbert space valued processes.

For further reference in this section, we explicitly write down our It\^{o} formula for the function $f(x) = \norm{x}_G^p$ for $p > 2$, and $X_t=\int_0^t \!\! \int_U \Phi(r, u)\, M(dr, du) $ where $\Phi \in \Lambda^{2}(T, M)$.

It is known that the first and second derivatives are given by the following linear and bilinear forms, respectively
\begin{align*}
    f_x(x)(g_1) & = p \norm{x}_G^{p-2} \inner{x} {g_1}_G, \\
    f_{xx}(x) (g_1, g_2) & = \begin{cases}
        p(p-2)\norm{x}_G^{p-4} \inner{x \otimes x (g_1)} {g_2}_G + p \norm{x}_G^{p-2} \inner{g_1}{g_2}_G, & x\neq 0, \\
        0, & x = 0.
    \end{cases}
\end{align*}
In this case we have
\begin{align*}
 \left\langle \! \left\langle f_{x}(s, X_{s-}), \Phi(s, u)\right\rangle \! \right\rangle_{G}(h)  = &\  \langle  f_{x}(s, X_{s-}), \Phi(s, u)h \rangle_{G^*,G} \\
  = &\  p \norm{X_{s-}}_G^{p-2} \inner{X_{s-}} {\Phi(s, u)h}_G \\
 = &\  p \norm{X_{s-}}_G^{p-2} \langle \! \langle X_{s-}, \Phi(s, u) \rangle \! \rangle_G(h) \end{align*}
 \begin{align*}
 \mbox{Tr}_{\Phi(s, u) Q_{M^c}^{1 / 2}} \left( f_{x x}(s, X_{s-}) \right) 
 = &\  p(p-2)\norm{X_{s-}}_G^{p-4} \mbox{Tr}_{\Phi(s, u) Q_{M^c}^{1 / 2}}( X_{s-} \otimes  X_{s-}) \\
 & + p \norm{X_{s-}}_G^{p-2} \norm{\Phi(s, u) Q_{M^c}^{1 / 2}}_{\mathcal{HS}(H,G)}^2.
\end{align*}
It\^{o} formula for this example takes the following form

\begin{align}
\norm{X_t}_{G}^p & = p \int_{0}^{t}\!\!\int_U  \norm{X_{s-}}_G^{p-2} \langle \! \langle X_{s-}, \Phi(s, u) \rangle \! \rangle_G M(ds, du)  \label{eqNormPStochInteg} \\
& +\frac{p(p-2)}{2}  \int_{0}^{t} \!\!\int_U \norm{X_{s-}}_G^{p-4} \mbox{Tr}_{\Phi(s, u) Q_{M^c}^{1 / 2}}( X_{s-} \otimes  X_{s-}) \operQuadraVari{M^c}(ds,du) \nonumber \\
& +\frac{p}{2} \int_{0}^{t} \!\!\int_U \norm{X_{s-}}_G^{p-2} \norm{\Phi(s, u) Q_{M^c}^{1 / 2}}_{\mathcal{HS}(H,G)}^2 \operQuadraVari{M^c}(ds,du) \nonumber \\
& +\sum_{0<s \leq t} \Bigl[\norm{X_{s}}^{p}_{G}-\norm{X_{s-}}^{p}_{G}-p\norm{X_{s-}}_{G}^{p-2}\inner{X_{s-}} {\Delta X_{s}}_G 
\Bigr]. \nonumber
\end{align}

As our first application, we establish a Burkholder inequality for the case where $X$ has continuous paths. Our proof is based on ideas from Theorem 2.140 of \cite{Lu_Zhang}, there for the case of It\^{o} process defined with respect to cylindrical Wiener processes.

\begin{theorem}[A Burkholder inequality; continuous case] \label{theoBurkholderContinuousCase}
Assume that for each $A\in\mathcal{A}$ and $h\in H$ the real-valued stochastic process $M(\cdot,A)(h)$ is continuous. Let $T>0$ and  $\Phi \in \Lambda^{2}(T, M)$. Put $X=(X_{t}: 0 \leq t \leq T)$ defined by
$$X_t=\int_0^t \!\! \int_U \Phi(r, u)\, M(dr, du).$$ 
For any $p>0$ and each $0 \leq t \leq T$,      
\begin{equation} \label{eqBurkholderContinuous} 
  \mathbb{E}\left[ \sup_{0 \leq s \leq t} \norm{ X_{s}}_{G}^p\right]  \leq C(p) \mathbb{E}\left[\left( \int_{0}^{t} \!\!\int_U  \norm{ \Phi(r,u)Q_M^{1/2}(r,u)}_{\mathcal{HS}(H,G)}^{2} \operQuadraVari{M}(dr,du) \right)^{p / 2}\right],
\end{equation}
where $C(p) = \left(\frac{p(p-1)}{2}\right)^{p/2} \left(\frac{p}{p-1}\right)^{(p-2)/2}$ for $p \geq 2$, and $C(p) = \frac{4 - p}{2 - p}$ for $0 < p < 2$.
 \end{theorem}
\begin{proof} 
If $p=2$, by the It\^{o} isometry, the result holds true with an equality for $C(p)=1$. Then we must only prove the result for the case $p \neq 2$.

\textbf{Claim:} It suffices to show that \eqref{eqBurkholderContinuous} holds true for $X$ bounded. 

To prove this claim, for each $k \in \N$, put
$$ \tau_{k}= \inf\{ s \in [0,t]: \norm{X_s}_{G} \geq k \},$$
with the convention that the infimum of the empty set is $t$.  

Let $\Phi_{k}=\mathbbm{1}_{[0,\tau_{k}]}\Phi$ and $X_{k}=\int_0^t\!\! \int_U \Phi_{k}(s, u) M(d s, d u)$. Then we have $\norm{X_{k}}_{G} \leq k$. If \eqref{eqBurkholderContinuous} holds true for $X_{k}$ we have
\begin{multline*}
\mathbb{E}\left[ \sup_{0 \leq s \leq t} \norm{ \int_0^s\!\! \int_U \Phi_{k}(r, u) M(d r, d u) }_{G}^p\right] \\ \leq  C(p) \mathbb{E}\left[\left( \int_{0}^{t} \!\!\int_U  \norm{ \Phi_{k}(r,u)Q_M^{1/2}(r,u)}_{\mathcal{HS}(H,G)}^{2} \operQuadraVari{M}(dr,du) \right)^{p / 2}\right] .   
\end{multline*}
Letting $k \rightarrow \infty$ in both sides of the above inequality, by Fatou's lemma we arrive at \eqref{eqBurkholderContinuous}. 

From now on we assume that $X$ is bounded. \medskip

\textbf{Case $p>2$:} By \eqref{eqNormPStochInteg} and since $X$ has continuous paths we have 
\begin{align*}
\norm{X_t}_{G}^p & = p \int_{0}^{t}\!\!\int_U  \norm{X_s}_G^{p-2} \langle \! \langle X_s, \Phi(s, u) \rangle \! \rangle_G M(ds, du)  \\
& +\frac{p(p-2)}{2}  \int_{0}^{t} \!\!\int_U \norm{X_s}_G^{p-4} \mbox{Tr}_{\Phi(s, u) Q_{M}^{1 / 2}}( X_s \otimes  X_s) \operQuadraVari{M}(ds,du)  \\
& +\frac{p}{2} \int_{0}^{t} \!\!\int_U \norm{X_s}_G^{p-2} \norm{\Phi(s, u) Q_{M}^{1 / 2}}_{\mathcal{HS}(H,G)}^2 \operQuadraVari{M}(ds,du) .
\end{align*}
However, since 
$$
\mbox{Tr}_{\Phi(s, u) Q_{M}^{1 / 2}}( X_s \otimes  X_s)
\leq \norm{\Phi(s, u) Q_{M}^{1 / 2}}_{\mathcal{HS}(H,G)}^2  \norm{X_s}^2,
$$
we arrive at the inequality:
\begin{align*}
\norm{X_t}_{G}^p & \leq p \int_{0}^{t}\!\!\int_U \norm{X_s}_G^{p-2} \langle \! \langle X_s, \Phi(s, u) \rangle \! \rangle_G M(ds, du)  \\
&+\frac{p(p-1)}{2} \int_{0}^{t} \!\!\int_U \norm{X_s}_G^{p-2} \norm{\Phi(s, u) Q_{M}^{1 / 2}}_{\mathcal{HS}(H,G)}^2 \operQuadraVari{M}(ds,du).
\end{align*}
Our boundedness assumption allows us to conclude that $$I_{t}=\int_{0}^{t}\!\!\int_U \norm{X_s}_G^{p-2} \langle \! \langle X_s, \Phi(s, u) \rangle \! \rangle_G \, M(ds, du)$$
is a $G$-valued continuous mean-zero  square integrable martingale. Thus, taking expectations we obtain
\begin{eqnarray*}
 \mathbb{E}\left[\|X_t\|^p_{G}\right] 
& \leq & \frac{p(p-1)}{2} \mathbb{E} \int_{0}^{t} \!\!\int_U \norm{X_s}_G^{p-2} \norm{\Phi(s, u) Q_{M}^{1 / 2}}_{\mathcal{HS}(H,G)}^2 \operQuadraVari{M}(ds,du) \\    
& \leq & \frac{p(p-1)}{2} \mathbb{E}\left( \sup _{0 \leq s \leq t}\|X_s\|_G^{p-2}\langle I\rangle_t\right).
\end{eqnarray*}
By H\"{o}lder's inequality and Doob's martingale inequality, we obtain 
\begin{eqnarray*}
 \mathbb{E}\left[\sup_{0 \leq s \leq t}\|X_s\|_G^{p-2}\langle I\rangle_t\right] 
& \leq & \left( \mathbb{E}\left[ \sup _{0 \leq s \leq t} \|X_s\|_{G}^{p} \right] \right)^{(p-2) / p}  \left( \mathbb{E}\left[\langle I\rangle_t^{p / 2}\right] \right)^{2 / p} \\
& \leq & \left(\frac{p}{p-1}\right)^{(p-2) / p}\left(\mathbb{E}\left[\|X_t\|_{G}^{p} \right]\right)^{(p-2) / p}\left(\mathbb{E}\left[\langle I\rangle_t^{p / 2}\right]\right)^{2 / p}.
\end{eqnarray*}

Let $D(p)=\frac{p(p-1)}{2} \cdot\left(\frac{p}{p-1}\right)^{(p-2)/p}$, then we have
$$
\mathbb{E}\left[\|X_t\|_{G}^{p} \right] \leq D(p)\left[\mathbb{E}\left(\|X_t\|_{G}^p\right)\right]^{1-\left(\frac{2}{p}\right)}\left(\mathbb{E}\left[ \langle I\rangle_t^{p / 2}\right]\right)^{2 / p}.  
$$
Therefore, by another use of Doob's inequality we obtain
$$
\mathbb{E}\left[ \sup _{0 \leq s \leq t} \|X_s\|_{G}^{p} \right] \leq \mathbb{E}\left[\|X_t\|_{G}^{p} \right]  \leq D(p)^{p / 2} \mathbb{E}\left[ \langle I\rangle_t^{p / 2}\right].
$$
Then we arrive at \eqref{eqBurkholderContinuous} with $C(p)=D(p)^{p / 2}$.\medskip

\textbf{Case $0<p<2$:}
The classical distribution identity for the expectation gives us
$$
\Exp\left[ \sup_{0 \leq s \leq t} \norm{X_s}_G^p \right] = p \int_0^{\infty} \lambda^{p-1} \Prob \left( \sup_{0 \leq s \leq t} \norm{X_s}_G^p > \lambda \right) \, d\lambda.
$$
By Proposition 6.19 in \cite{CCFM:SPDE} (or rather its proof), with $a = \lambda$ and $b = \lambda^2$, we have
\begin{equation*}
\Prob \left( \sup_{0 \leq s \leq t} \norm{X_s}_{G} > \lambda \right) \leq \frac{1}{\lambda^{2}} \min\left( \lambda^2, \Exp[\langle I\rangle_t]\right) +  \Prob \left(  \langle I\rangle_t > \lambda^2 \right).   
\end{equation*}
Then we have
\begin{align*}
  \Exp\left[ \sup_{0 \leq s \leq t} \norm{X_s}_G^p \right] &\leq p \int_0^{\infty} \lambda^{p-3} \min\left( \lambda^2, \Exp[\langle I\rangle_t]\right) \, d\lambda + p \, \int_0^{\infty} \lambda^{p-1} \Prob \left(  \langle I\rangle_t > \lambda^2 \right) d\lambda \\
& = I_1 + I_2.  
\end{align*}

For $I_2$, we obtain the desired estimate, since by a simple change of variable we get
$$
I_2 = \Exp[\langle I\rangle_t^{p/2}] = \Exp\left[ \left( \int_0^t \!\! \int_U \| \Phi(s,u)\circ Q_M^{1/2}\|^2_{\mathcal{HS}(H,G)} \, \operQuadraVari{M}(ds, du)  \right)^{p/2} \right].
$$
For $I_1$, by considering the cases where $\lambda^2 \leq \langle I\rangle_t$ and $\lambda^2 > \langle I\rangle_t$ separatedly we obtain
\begin{align*}
    I_1 & = p \Exp\left[   \int_0^{\infty} \lambda^{p-3} \min(\lambda^2, \langle I\rangle_t)\, d\lambda \right] = p \Exp\left[ \int_0^{\sqrt{\langle I\rangle_t}} \lambda^{p-1} \, d\lambda \right] + p \Exp\left[ \int_{\sqrt{\langle I\rangle_t}}^{\infty} \lambda^{p-3} \langle I\rangle_t \, d\lambda \right] \\
    & = \Exp[ \langle I\rangle_t^{p/2}] + \frac{p}{2-p}\Exp[\langle I\rangle_t^{p/2}] = \frac{2}{2-p} \Exp\left[ \langle I\rangle_t^{p/2} \right],
\end{align*}
which is the desired inequality.
\end{proof}

Now we move on the case where the paths of the cylindrical martingale-valued measure $M$ are purely discontinuous. In this case we base our argument on proof of Theorem 2.3 of \cite{ZhuBrzezniakLiu:2019}, there for the case of Poisson random measures.

\begin{theorem}[Burkholder's inequality; purely discontinuous case] \label{theoBurkholderDiscontinuousCase}
Assume that for each $A\in\mathcal{A}$ and $h\in H$ the real-valued stochastic process $M(\cdot,A)(h)$ is purely discontinuous.  Let $T>0$ and  $\Phi \in \Lambda^{2}(T, M)$. Put $X=(X_{t}: 0 \leq t \leq T)$ defined by 
$$X_t=\int_0^t \!\! \int_U \Phi(r, u)\, M(dr, du).$$ 
Then for any $p >0$ and for each $0 \leq t \leq T$,      
\begin{equation}\label{eqBurkholderPurelyDiscontinuous}
\mathbb{E}\left[ \sup_{0 \leq s \leq t} \norm{ X_{s} }_{G}^p\right] \leq  C(p) \Exp \left(  [X]^{p/2}_{t}  \right)  =C(p)  \Exp \left[ \left(   \sum_{0<s \leq t} \norm{\Delta X_{s}}^{2}_G \right)^{p/2} \right],   
\end{equation} 
where $C(p) = \frac{4 - p}{2-p}$, for $0<p<2$, $C(p)=1$ for $p = 2$ and $C(p) = (2p(p-1))^{p/2} \left( \frac{p}{p-1} \right)^{p^{2}/2}$ for $p > 2$.
\end{theorem}
\begin{proof} 
As in the proof of Theorem \ref{theoBurkholderContinuousCase}, it suffices to show the result for $p \neq 2 $ and under the assumption that   \eqref{eqBurkholderPurelyDiscontinuous} holds true for $X$ bounded. Moreover, the proof for $0 < p <2$ in Theorem \ref{theoBurkholderContinuousCase} remains valid in this case. Hence, we only need to check the case $p>2$. 

For $p>2$, by \eqref{eqNormPStochInteg} we have
\begin{align}
\norm{X_t}_{G}^p & = p \int_{0}^{t}\!\!\int_U  \norm{X_{s-}}_G^{p-2} \langle \! \langle X_{s-}, \Phi(s, u) \rangle \! \rangle_G M(ds, du)  \nonumber \\
& +\sum_{0<s \leq t} \Bigl[\norm{X_{s}}^{p}_{G}-\norm{X_{s-}}^{p}_{G}-p\norm{X_{s-}}_{G}^{p-2}\inner{X_{s-}} {\Delta X_{s}}_G 
\Bigr]. \nonumber\\
& = I_{1}(t)+I_{2}(t). \nonumber
\end{align}
Since $X$ is a $G$-valued martingale and $I_{1}$ has zero mean, by Doob's inequality we have
\begin{equation}\label{eqIneqSupOrderPBirkholderDiscontinuous}
\Exp\left[ \sup_{0 \leq s \leq t} \norm{X_{s}}_G^p \right] \leq \left( \frac{p}{p-1} \right)^{p} \Exp \left[ \norm{X_{t}}_G^p \right]   = \left( \frac{p}{p-1} \right)^{p} \Exp \left[  I_{2}(t) \right]  .
%= \left( \frac{p}{p-1} \right)^{p} \Exp \left[  I_{1}(t)+I_{2}(t) \right]
\end{equation}
Now, as a consequence of  the mean value theorem (see e.g. Corollary 12.2.9 in \cite{BogachevSmolyanov}) we have
\begin{multline}\label{eqInequaMeanValueBurkholder}
\abs{\norm{X_{s}}^{p}_{G}-\norm{X_{s-}}^{p}_{G}-p\norm{X_{s-}}_{G}^{p-2}\inner{X_{s-}} {\Delta X_{s}}_G } \\
\leq \sup_{\theta \in (0,1)} \norm{p\norm{X_{s-}+\theta \Delta X_{s}}_{G}^{p-2}\inner{X_{s-}} {\cdot }_G - p\norm{X_{s-}}_{G}^{p-2}\inner{X_{s-}} {\cdot}_{G}}_{\mathcal{L}(G,\R)} \norm {\Delta X_{s}}_G     
\end{multline}
Observe that for all $0 \leq s \leq t$, 
$$ \norm{X_{s-}}_{G} \leq \sup_{0 \leq s \leq t} \norm{X_{s-}}_{G} \leq \sup_{0 \leq s \leq t} \norm{X_{s}}_{G}. $$
Also, since $X_{s} =X_{s-}+\Delta X_{s}$, we get 
$$ \norm{X_{s-}+\Delta X_{s} }_{G} \leq \sup_{0 \leq s \leq t} \norm{X_{s}}_{G}. 
 $$
By using the inequality $\norm{x+\theta y}_{G} \leq \max\{\norm{x}_{G}, \norm{x+y}_{G}\}$ for $0 < \theta < 1$ and $x,y \in G$, we have 
$$ \norm{X_{s-}+\theta \Delta X_{s}}_{G} \leq \max \left\{ \norm{X_{s-}}_{G}, \norm{X_{s-}+\Delta X_{s} }_{G} \right\} \leq \sup_{0 \leq s \leq t} \norm{X_{s}}_{G}.  $$
On the other hand, since the Hilbert space $G$ is 2-smooth, the (Fr\'echet) derivative of $\psi_{p}(x)\defeq \norm{x}^{p}_{G}$ is locally H\"{o}lder continuous on $G$, that is, 
$$ \norm{\psi'_{p}(x)-\psi'_{p}(y)}_{\mathcal{L}(G,\R)} \leq D_{p} \left( \norm{x}_{G} +\norm{y}_{G} \right)^{p-2} \norm{x-y}_{G}, \quad \forall x,y \in G, $$
where $D_{p}= p (p-1)$ (see for instance the proof of Lemma 2.1 in \cite{vanNeervenZhu:2011}). Then, we have
\begin{flalign*}
& \norm{p\norm{X_{s-}+\theta \Delta X_{s}}_{G}^{p-2}\inner{X_{s-}} {\cdot }_G - p\norm{X_{s-}}_{G}^{p-2}\inner{X_{s-}} {\cdot}_{G}}_{\mathcal{L}(G,\R)}   \\
& \leq D_{p} \left( \norm{X_{s-}+\theta \Delta X_{s}}_{G} + \norm{X_{s-}}_{G}  \right)^{p-2} \norm{\Delta X_{s}}_G \\
& \leq 2 D_{p} \left( \sup_{0 \leq s \leq t} \norm{X_{s}}_{G}^{p-2} \right) \norm{\Delta X_{s}}_G  
\end{flalign*}
Thus, by  \eqref{eqInequaMeanValueBurkholder} we get 
$$ \abs{\norm{X_{s}}^{p}_{G}-\norm{X_{s-}}^{p}_{G}-p\norm{X_{s-}}_{G}^{p-2}\inner{X_{s-}} {\Delta X_{s}}_G } 
\leq D_{p} \sup_{0 \leq s \leq t} \norm{X_{s}}_{G}^{p-2} \norm{\Delta X_{s}}^{2}_G  $$
Hence, by the above calculations, \eqref{eqOptionalQuadraticVariationIntegral} and H\"{o}lder inequality, we have
\begin{eqnarray*}
 \Exp \left[  I_{2}(t) \right]   
 & \leq & \Exp \left[  \sum_{0<s \leq t} \abs{ \norm{X_{s}}^{p}_{G}-\norm{X_{s-}}^{p}_{G}-p\norm{X_{s-}}_{G}^{p-2}\inner{X_{s-}} {\Delta X_{s}}_G } \right]   \\
 & \leq & 2 D_{p} \Exp \left[ \left( \sup_{0 \leq s \leq t} \norm{X_{s}}_{G}^{p-2} \right) \cdot  \sum_{0<s \leq t} \norm{\Delta X_{s}}^{2}_G  \right] \\
% & = & 2 D_{p} \Exp \left[ \left( \sup_{0 \leq s \leq t} \norm{X_{s}}_{G}^{p-2} \right) \cdot  [X]_{t}  \right] \\
 & \leq & 2 D_{p} \left(  \Exp \left[ \sup_{0 \leq s \leq t} \norm{X_{s}}_{G}^{p} \right]\right)^{(p-2)/p} \left( \Exp \left[  [X]_{t}^{p/2}  \right] \right)^{2/p} 
\end{eqnarray*}
Thus, by \eqref{eqIneqSupOrderPBirkholderDiscontinuous} it follows that 
$$
\Exp\left[ \sup_{0 \leq s \leq t} \norm{X_{s}}_G^p \right]  \leq 2 D_{p} \left( \frac{p}{p-1} \right)^{p} \left( \Exp \left[ \sup_{0 \leq s \leq t} \norm{X_{s}}_{G}^{p} \right]\right)^{1-2/p} \left( \Exp \left[  [X]_{t}^{p/2}  \right] \right)^{2/p}.
$$
and we arrive at \eqref{eqBurkholderPurelyDiscontinuous} with $C(p)=2^{p/2} D_{p}^{p/2} \left( \frac{p}{p-1} \right)^{p^{2}/2} = (2p(p-1))^{p/2} \left( \frac{p}{p-1} \right)^{p^{2}/2}$. 
\end{proof}

We are ready to prove Burkholder's inequality for a (general) orthogonal cylindrical martingale-valued measure $M$. 

\begin{theorem}[Burkholder's inequality] \label{theoBurkholdersInequalityGeneral}
Let $T>0$ and  $\Phi \in \Lambda^{2}(T, M)$. Let $X=(X_{t}: 0 \leq t \leq T)$ be defined by 
$$X_t=\int_0^t \!\! \int_U \Phi(r, u)\, M(dr, du).$$ 
Then for any $p >0$ there exists $C(p)>0$ such that for each $0 \leq t \leq T$,      
\begin{eqnarray*}
\mathbb{E}\left[ \sup_{0 \leq s \leq t} \norm{ X_{s}}_{G}^p\right] 
& \leq &  C(p) \mathbb{E}\left[\left( \int_{0}^{t} \!\!\int_U  \norm{ \Phi(r,u)Q_{M^{c}}^{1/2}(r,u)}_{\mathcal{HS}(H,G)}^{2} \operQuadraVari{M^{c}}(dr,du) \right)^{p / 2}\right] \\
& + & C(p)  \Exp \left[ \left(   \sum_{0<s \leq t} \norm{\Delta X_{s}}^{2}_G \right)^{p/2} \right] 
\end{eqnarray*}
\end{theorem}
\begin{proof}
By \eqref{eqDecompStochIntegral}, and using the inequality $(a + b)^p \leq \max\{ 2^{p-1}, 1 \}(a^p + b^p)$, for $a,b \geq 0$ and $p > 0$, we have
\begin{flalign*}
& \mathbb{E}\left[\norm{ \int_0^t\!\! \int_U \Phi(s, u) M(d s, d u) }_{G}^p\right] \\ & \leq C(p) \left\{ \mathbb{E}\left[\norm{ \int_0^t \!\! \int_U \Phi(s, u)\, \, M^{c}(ds, du) }_{G}^p\right] +  \mathbb{E}\left[\norm{  \int_0^t \!\! \int_U \Phi(s, u) \,  \, M^{d}(ds, du) }_{G}^p\right]
\right\} 
\end{flalign*}
The result is now a consequence of Theorems \ref{theoBurkholderContinuousCase} and \ref{theoBurkholderDiscontinuousCase}. 
\end{proof}

\begin{remark}
As in Theorem \ref{theoBurkholderDiscontinuousCase}, by Equation \eqref{eqOptionalQuadraticVariationIntegral}, the right hand side of the inequality in Theorem \ref{theoBurkholderContinuousCase} coincides with $\Exp \left(  [X]^{p/2}_{t}  \right)$.  By using the relation $a^p + b^p \leq \max\{ 2^{1-p} , 1 \}(a + b)^p$, for $a, b \geq 0$ and $p >0$, we can write Burkholder's inequality in the form
\begin{equation}
\label{Burkholdergeneralintermsofoptionalvariation}
    \mathbb{E}\left( \sup_{0 \leq s \leq t} \norm{ X_{s}}_{G}^p\right) \leq C(p)\Exp \left(  [X]^{p/2}_{t}  \right)
    \end{equation}
\end{remark}

\begin{example}
	\label{examHvaluedLevy}
Let $L=(L_{t}: t \geq 0)$ be an $H$-valued c\`{a}dl\`{a}g L\'{e}vy process with  L\'{e}vy-It\^{o} decomposition
\begin{equation}\label{eqLevyItoDecomp}
L_{t}= t \xi +W_{t}+\int_{\norm{h} <1} \, h \,  \widetilde{N}(t,dh)+ \int_{\norm{h} \geq 1} \, h \, N(t,dh). 
\end{equation}
Here $N$ is the Poisson random measure associated to $L$ with corresponding L\'{e}vy measure $\lambda$ and $\tilde{N}$ the corresponding compensated Poisson random measure. Let $Q$ denote the positive trace class operator corresponding to the covariance operator of the Wiener process $W=(W_{t}:t \geq 0)$. 

Let $U \in \mathcal{B}(H)$ be such that $0 \in U$ and $\int_{U} \, \norm{u}^{2} \lambda(du)< \infty$. Take $ \mathcal{A}= \{ A \subseteq U : A - \{0\} \textup{ is bounded below}\}$ and let $M=(M(t,A): t \geq 0, A \in \mathcal{A})$ be given by
\begin{equation} \label{levyMartValuedMeasExam} 
M(t,A) = W_{t} \delta_{0}(A) + \int_{A \backslash \{0 \}} \,  u \, \widetilde{N}(t,du), \quad \forall \, t \geq 0, \, A \in \mathcal{A}. 
\end{equation}
We know by Example 5.24 in \cite{CCFM:SPDE} that $M$ is an orthogonal cylindrical martingale-valued measure with quadratic variation given by 
\begin{equation*}\label{eqQuadraticVariationLevyCMVM}
\operQuadraVari{M}((s,t] \times A) = (t-s) \left[ \norm{Q} \delta_{0}(A) + \int_{A \backslash \{0\}} \, \norm{u}^{2} \, \lambda(du) \right].  
\end{equation*}
Moreover, by Example 6.11  in \cite{CCFM:SPDE} the square integrability condition for integrands with respect to $M$ is given by 
\begin{equation}\label{eqSquaIntegrIntegrandLevyCMVM}
 \norm{\Phi}^2_{\Lambda^2(M,T)} = \Exp \int_{0}^{T}  \| \Phi(s,0)\circ Q^{1/2}\|^2_{\mathcal{HS}(H,G)} \, ds  + \Exp \int_{0}^{T}\!\! \int_{U} \| \Phi(s,u) u \|^2_{G} \, \lambda(du) ds< \infty.    
\end{equation}

Now, by the independence of $N$ and $W$, we can easily check that the continuous and purely discontinuous parts of $M$ are given by 
$$ M^{c}(t,A) = W_{t} \delta_{0}(A), \quad  M^{d}(t,A) = \int_{A \backslash \{0 \}} \,  u \, \widetilde{N}(t,du),$$
with quadratic variations
$$ \operQuadraVari{M^{c}}((s,t] \times A) = (t-s)  \norm{Q} \delta_{0}(A), \quad \operQuadraVari{M^{d}}((s,t] \times A)   = (t-s)\int_{A \backslash \{0\}} \, \norm{u}^{2} \, \lambda(du).   $$
Moreover, we have
$$ Q_{M^{c}}(\omega, r, u)h = \frac{Q h}{\norm{Q}} \mathbbm{1}_{\{0\}}(u), \quad Q_{M^{d}}(\omega, r, u)h = \frac{(u,h)_{H} \,  u}{\norm{u}^2}  \mathbbm{1}_{U \setminus \{0\}}(u), $$
Now let  $ \displaystyle X_t=\int_0^t \!\! \int_U \Phi(r, u)\, M(dr, du)$ for $\Phi$ satisfying \eqref{eqSquaIntegrIntegrandLevyCMVM}. Observe that by the definition of $M$ in \eqref{levyMartValuedMeasExam} and by Proposition 6.16 in \cite{CCFM:SPDE} we have $\Delta X_{s}(\omega) = \Phi(\omega, s, \Delta L_{s}) \Delta L_{s}$. Hence, 
$$  \sum_{0<s \leq t} \norm{\Delta X_{s}(\omega)}^{2}_G  = \int_{0}^{t}\!\! \int_{U \setminus \{0\} } \, \norm{ \Phi(\omega, s, u) u}_{G}^{2} N(ds,du). $$
This way, for $p \geq 2$ we have by Theorem \ref{theoBurkholdersInequalityGeneral},
\begin{eqnarray*}
\mathbb{E}\left[ \sup_{0 \leq s \leq t} \norm{ X_{s}}_{G}^p\right] 
& \leq &  C(p) \mathbb{E}\left[\left( \int_{0}^{t}  \| \Phi(s,0)\circ Q^{1/2}\|^2_{\mathcal{HS}(H,G)} \, ds  \right)^{p / 2}\right] \\
& + & C(p)  \Exp \left[ \left(  \int_{0}^{t} \!\! \int_{U \setminus \{0\}} \, \norm{ \Phi(s, u) u}_{G}^{2} N(ds,du) \right)^{p/2} \right]
\end{eqnarray*}

Taking into account that 
\begin{eqnarray*}
    \int_{0}^{t}\!\! \int_{U \setminus \{0\}} \, \norm{ \Phi(s, u) u}_{G}^{2} N(ds,du) & = & \int_{0}^{t}\!\! \int_{U \setminus \{0\}} \, \norm{ \Phi(s, u) u}_{G}^{2} \widetilde{N}(ds,du) \\ 
    &+& \int_{0}^{t}\!\! \int_{U \setminus \{0\}} \, \norm{ \Phi(s, u) u}_{G}^{2} \lambda(du)\, ds,
\end{eqnarray*}
we obtain
\begin{flalign*}
    & \Exp\left[  \left( \int_{0}^{t}\!\! \int_{U \setminus \{0\}} \, \norm{ \Phi(s, u) u}_{G}^{2} N(ds,du) \right)^{p/2} \right] \\ & \leq \widetilde{C}(p) \Exp\left[  \left( \int_{0}^{t}\!\! \int_{U \setminus \{0\}} \, \norm{ \Phi(s, u) u}_{G}^{2} \widetilde{N}(ds,du) \right)^{p/2} \right] \\
    & + \widetilde{C}(p)\Exp\left[  \left( \int_{0}^{t}\!\! \int_{U \setminus \{0\}} \, \norm{ \Phi(s, u) u}_{G}^{2} \lambda(du)\, ds \right)^{p/2} \right].
\end{flalign*}
Then we can apply Corollary 2.4 in \cite{ZhuBrzezniakLiu:2019} to the compensated Poisson integral to deduce a Kunita inequality (see Theorem 4.4.23 in \cite{ApplebaumLPSC} for the finite-dimensional case)
\begin{eqnarray*}
\mathbb{E}\left[ \sup_{0 \leq s \leq t} \norm{ X_{s}}_{G}^p\right] 
& \leq &  C(p)  \left\{ \mathbb{E}\left[\left( \int_{0}^{t}  \| \Phi(s,0)\circ Q^{1/2}\|^2_{\mathcal{HS}(H,G)} \, ds  \right)^{p / 2}\right] \right. \\
& + &   \Exp \left(  \int_{0}^{t}\!\! \int_{U \setminus \{0\}} \, \norm{ \Phi(s, u) u}_{G}^{p} \lambda(du) ds \right) \\
& + & \left. \Exp \left[ \left(  \int_{0}^{t} \!\! \int_{U \setminus \{0\}} \, \norm{ \Phi(\omega, s, u) u}_{G}^{2} \lambda(du) ds\right)^{p/2} \right] \right\}. 
\end{eqnarray*}
\end{example}

\section{It\^{o}'s formula for particular examples of cylindrical martingale-valued measures}\label{sectionAppliItoToHValuedMVM}

In this section we show how cylindrical orthogonal martingale-valued measures generalize both Hilbert-space valued martingale measures and cylindrical square integrable martingales, and apply the results of Theorems \ref{theoItoFormulaScalarCase} and  \ref{theoItoFormulaVectorCase} to obtain It\^{o}'s formulas in each case. Some of  the results we obtain are new in the literature, and in some other cases we obtain alternative proofs for some known results.

\subsection{Hilbert-space valued orthogonal martingale-valued measures}

Assume  that $M$ is an $H$-valued orthogonal martingale-valued measure (see Definition 3.3 in \cite{CCFM:SPDE}).  By Theorem 3.6 in \cite{CCFM:SPDE}, it has an intensity measure $\nu_{M}$. Assume further that for every $A, B \in \mathcal{A}$ disjoint and $h \in H$, the real-valued martingales $(M(A),h)_{H}$ and $(M(B),h)_{H}$ are orthogonal. Then by Theorem 4.6 in \cite{CCFM:SPDE}, $M$ induces a  cylindrical orthogonal martingale-valued measure $\widehat{M}=(\widehat{M}(t,A): t \geq 0, A \in \mathcal{A})$ on $H$ given by
\begin{equation*}\label{eqInducedCMVMDefinedByHValuedMVM}
\widehat{M}(t,A)(\omega)(h)\defeq (M(t,A)(\omega),h)_{H}, \quad \forall \omega \in \Omega, t \geq 0, A \in \mathcal{A}, h \in H.    
\end{equation*}  
For simplicity, we allow ourselves an abuse of notation and also denote $\widehat{M}$ by $M$. It is easy to check (using for example (2.2) in \cite{CCFM:SPDE}) that 
\begin{equation}\label{eqInequaliIntensityHValued}
 \nu_{h} \leq \norm{h}^{2} \nu_{M}, \quad \forall h \in H. 
\end{equation}

\begin{proposition}\label{propExistenQuadrVariaHilbertMVM}
Let $M$ as described above, and assume that the intensity measure  $\nu_{M}$ of $M$ is bounded on $[0,t] \times \mathcal{A}$ for every $t>0$, that is
$$ \sup_{A \in \mathcal{A}} \nu_{M}([0,t]\times A)< \infty. $$
Then, the family of intensity measures $(\nu_{h}: \norm{h}_{H}=1)$ of $M$ satisfies the sequential boundedness property and (consequently) $M$ has a unique quadratic variation satisfying $\operQuadraVari{M} \leq \nu_{M}$ and an operator quadratic variation $Q_{M}: \Omega \times [0,T]  \times U \rightarrow \mathcal{L}(H,H)$ satisfying \eqref{existenceofQ}. 
\end{proposition}
\begin{proof}
First, it is a consequence of \eqref{eqInequaliIntensityHValued} that  $M$ possesses at least one quadratic variation as a cylindrical martingale-valued measure.
To prove uniqueness, it suffices to check that the family of intensity measures $(\nu_{h}: \norm{h}_{H}=1)$ satisfies the sequential boundedness property. 

To do this, we will need the following inequality which is valid for any two real-valued square integrable martingales $m$ and $n$:
$$ \abs{ \quadraVari{m}_{t}- \quadraVari{n}_{t}} \leq \sqrt{\quadraVari{m-n}_{t}}\left( \sqrt{\quadraVari{m}_{t}}+ \sqrt{\quadraVari{n}_{t}} \right).$$
A proof can be carried out using basic properties of the quadratic variation and Kunita's inequality. 
% (which can be proved by modifying the arguments in the proof of Lemma 4.77 in \cite{KarandikarRao:2018})
Now, for $h_{1}, h_{2} \in H$, $t \geq 0$ and $A \in \mathcal{A}$,  we apply the above inequality with $m_{t}=M(t,A)h_{1}$ and $n_{t}=M(t,A)h_{2}$, and use \eqref{eqInequaliIntensityHValued} to obtain  
\begin{flalign*}
& \abs{\nu_{h_{1}}([0,t]\times A) - \nu_{h_{2}}([0,t]\times A)} \\
&=  \abs{\quadraVari{M(A)h_{1}}_{t} - \quadraVari{M(A)h_{2}}_{t}} \\
& \leq \sqrt{\quadraVari{M(A)(h_{1}-h_{2})}}_t  \left( \sqrt{\quadraVari{M(A)h_{1}}_{t}}+ \sqrt{\quadraVari{M(A)h_{2}}_{t}} \right) \\
& \leq \norm{h_{1}-h_{2}} \left( \norm{h_{1}}+\norm{h_{2}} \right) \nu_{M}([0,t]\times A)
\end{flalign*}
Hence, if $h_{n} \rightarrow h$ in the unit sphere in $H$ we have 
\begin{equation}\label{eqContinuityIntensityMeasures}
\sup_{A \in \mathcal{A}} \sup_{0 \leq s \leq t} \abs{ \nu_{h_{n}}((s,t] \times A) - \nu_{h}((s,t] \times A)} \rightarrow 0, \quad \mbox{as } n \rightarrow \infty. 
\end{equation}
Then, by  Proposition 5.7 in \cite{CCFM:SPDE} the family of intensity measures $(\nu_{h}: \norm{h}_{H}=1)$ satisfies the sequential boundedness property.  
Now, Theorem 5.8 in \cite{CCFM:SPDE} guarantees the existence of a unique quadratic variation $\operQuadraVari{M} \leq \nu_{M}$. Finally,  the condition \eqref{eqBoundedrangequadraticvariation} is satisfied thanks to our assumption on $\nu_{M}$, and so $M$ has an operator quadratic variation $Q_{M}: \Omega \times [0,T]  \times U \rightarrow \mathcal{L}(H,H)$ satisfying \eqref{existenceofQ}. 
\end{proof}

Thanks to Proposition \ref{propExistenQuadrVariaHilbertMVM} and Theorem  \ref{theoItoFormulaVectorCase} we obtain the following version of It\^{o} formula. 

\begin{corollary}
\label{coroItoHvalued}
Let $M$ be an $H$-valued orthogonal martingale-valued measure with the property that for every $A, B \in \mathcal{A}$ disjoint and $h \in H$, the real-valued martingales $(M(A),h)_{H}$ and $(M(B),h)_{H}$ are orthogonal. Assume that the intensity measure  $\nu_{M}$ of $M$ is bounded on $[0,t] \times \mathcal{A}$ for every $t>0$.   
Let $\xi$, $\beta$, $\psi$, $\Phi$ and $X$ be as in Definition \ref{defiItoProcess}. Let $f:[0, \infty) \times G \rightarrow K$ be a function in the $C^{1,2}$ class, such that $f_{xx}$ is uniformly continuous on bounded subsets of $[0,\infty)\times G$. Then  almost surely we have for every $t\geq 0$
$$
\begin{aligned}
& f(t, X_t) = f(0, \xi)+ \int_{0}^{t} f_{t}(s, X_{s-}) d s +\int_{0}^{t}\!\!\int_U  f_{x}(s, X_{s-}) \Phi(s, u)  M(ds, du) \\
& + \int_{0}^{t} f_{x}(s, X_{s-}) \psi(s) d \beta(s)   +\frac{1}{2} \int_{0}^{t} \!\!\int_U 
\mbox{Tr}_{\Phi(s, u) Q_{M^{c}}^{1 / 2}} \left( f_{x x}(s, X_{s-}) \right)
\operQuadraVari{M^{c}}(ds,du) \\
 & +\sum_{0<s \leq t}\left[f\left(s,X_s\right)-f\left(s,X_{s-}\right)-f_x\left(s,X_{s-}\right)\left(\Delta X_{s} \right)\right]
\end{aligned}
$$
\end{corollary}

\begin{remark}
As far as we know, It\^{o}'s formula  for Hilbert-space valued martingale-valued measures had never been proved.   This can be applied to recover the classical It\^o formula for Wiener processes (as in \cite{DaPratoZabczyk}), for square integrable L\'evy processes (as in \cite{PeszatZabczyk}) or square integrable martingales (as in \cite{MetivierPellaumail}).
\end{remark}

\begin{example}
Assume $N$ is a Poisson random measure on $\R \times \R^{d}$ with characteristic (intensity) measure $\mbox{Leb}\times \lambda$, where $\lambda$ is a $\sigma$-finite measure on $\R^{d}$. Take $U=\R^{d}$ and $\mathcal{A}=\mathcal{B}(\R^{d})$. Let  $\eta:\Omega \times \R_{+} \times \R^{d} \rightarrow H$ be a predictable process for which
$$ \Exp \int_{0}^{T}\!\! \int_{A} \norm{\eta(t,u)}^{2} \lambda(du) dt<\infty.$$
Then, $M$ given by
$$M(t,A)=\int_{0}^{t}\!\!  \int_{A} \eta(r,u) \widetilde{N}(dr,du)$$
defines an $H$-valued orthogonal martingale-valued measure with purely discontinuous paths, hence $\operQuadraVari{M^{c}}=0$ and 
$$ \operQuadraVari{M^{d}}=\int_{0}^{t}\!\!  \int_{A} \norm{\eta(r,u)}^{2} \lambda(du) dr. $$
In this case, It\^{o}'s formula in Corollary \ref{coroItoHvalued} takes the form:
\begin{align*}
&f(t, X_t)  = f(0, \xi)+ \int_{0}^{t} f_{t}(s, X_{s-}) d s  + \int_{0}^{t} f_{x}(s, X_{s-}) \psi(s) d \beta(s) \\
& \quad
+\int_{0}^{t}\!\!\int_U  f_{x}(s, X_{s-}) \Phi(s, u)  \eta(s,u) \widetilde{N}(ds,du)\\
& \quad  +
\int_{0}^{t}\!\! \int_{U} \left[f(s,X_{s-}+\Phi(s,u)\eta(s,u)) - f(s,X_{s-})-f_{x}(s,X_{s-})( \Phi(s, u) \eta(s,u))\right] \, N(ds,du),
\end{align*}

where $\xi$, $\beta$, $\psi$, $\Phi$, $X$ and $f$ are as in Corollary \ref{coroItoHvalued} . 
    
\end{example}

\subsection{Cylindrical square integrable martingales}

Assume that $Z=(Z_{t}: t \geq 0)$ is a cylindrical \emph{continuous} zero-mean square integrable  martingale such that for each $t \geq 0$ the mapping $Z_{t}: H \rightarrow L^{0}(\Omega,\mathcal{F},\Prob)$ is continuous. Consider a one point set $U=\{a\}$, $\mathcal{A}=2^{U}$ and let $M=(M(t,A): t \geq 0, A \in \mathcal{A})$ be defined by 
\begin{equation}\label{eqCMVMContiyliMartingale}
M(t,A)(h)=Z_{t}(h) \delta_{a}(a), \quad \, \forall h \in H, \, t \geq 0, A \in \mathcal{A}.
\end{equation}
We know, by Examples 4.4 and 5.2 in \cite{CCFM:SPDE}, that $M$ is a cylindrical orthogonal martingale-valued measure with family of intensity measures $\nu_h(ds,du)=\lambda_{\quadraVari{Z(h)}}(ds)\delta_{a}(du) $,
where $\lambda_{\quadraVari{Z(h)}}$ denotes the Lebesgue-Stieltjes measure associated to $\quadraVari{Z(h)}$. By Example 5.13, in \cite{CCFM:SPDE} the family of intensity measures  $(\nu_{h}: \norm{h}=1)$ of $M$ satisfies \eqref{eqContinuityIntensityMeasures}, hence the sequential boundedness property by Proposition 5.7 in \cite{CCFM:SPDE}. 
Thus, there exists a unique quadratic variation given by  $\operQuadraVari{M}= \gamma \otimes \delta_{a}  $ $\Prob$-a.e., for $\gamma=\sup_{n \geq 1} \lambda_{\quadraVari{Z(h_{n})}}$ where $(h_n: n \in \N)$ is dense in the unit sphere in $H$. Since the condition \eqref{eqBoundedrangequadraticvariation} is trivially satisfied, then $M$ has an operator quadratic variation $Q_{M}: \Omega \times [0,T]  \times U \rightarrow \mathcal{L}(H,H)$ satisfying \eqref{existenceofQ}. 

Let $Q_{Z}: \Omega \times [0,T] \rightarrow \mathcal{L}(H,H)$ be given by $Q_{Z}(\omega,t)=Q_{M}(\omega, t,a)$.  
We say that $\Phi:\Omega \times [0,T] \rightarrow \mathcal{HS}(H_{Q_{Z}},G)$ is in $\Lambda^{2}(Z,T)$, if $ (\omega,t) \mapsto \Phi(\omega,t) \circ Q^{1/2}_{Z}(\omega,t)(h)$ is  $\mathcal{P}_{T}/\mathcal{B}(G)$-measurable and satisfies
$$ \Exp \int_{0}^{T} \| \Phi(t)\circ Q_{Z}^{1/2}(t)\|^2_{\mathcal{HS}(H,G)} \, \gamma(dt) < \infty. $$
One can easily check from the definition of the stochastic integral in  \eqref{NewDefIntSimpleIntegrand}, that if $\Phi \in \Lambda^{2}(Z,T)$, then $\Phi \in \Lambda^{2}(M,T)$ and 
\begin{equation}\label{eqCompatibleintegrals}
\int_0^t \!\! \int_U \Phi(s)\, M(ds, du) = \int_0^t \Phi(s)\, dZ_{s},    
\end{equation} 
where the integrand in the right-hand side is the (radonifying) integral for cylindrical square integrable martingales as in \cite{MetivierPellaumail} (or as in \cite{DaPratoZabczyk, GawareckiMandrekar} for cylindrical Wiener processes). 

From the above observation and by an application of  Corollary \ref{coroItoVectorContinuousCase} we obtain an It\^{o} formula for cylindrical continuous square integrable martingales. This result coincides with the one obtained by Veraar and Yaroslatsev in (\cite{VeraarYaroslavtsev:2016}, Theorem 4.16) for the Hilbert space setting. 

\begin{corollary}
Let $\xi$, $\beta$, $\psi$ as in Definition \ref{defiItoProcess}. Assume that $Z=(Z_{t}: t \geq 0)$ is a cylindrical continuous zero-mean square integrable  martingale such that for each $t \geq 0$ the mapping $Z_{t}: H \rightarrow L^{0}(\Omega,\mathcal{F},\Prob)$ is continuous. Let $\Phi \in \Lambda^{2}(Z,T)$ and consider  the It\^{o} process:  
$$X_t : = \xi + \int_0^t \psi(s) d\beta(s) + \int_0^t \Phi(s)\, dZ_{s}. $$
 Let $f:[0, \infty) \times G \rightarrow K$ be a function of class $C^{1,2}$ such that $f_{xx}$ is uniformly continuous on bounded subsets of $[0,\infty)\times G$. Then,  almost surely we have for every $t\geq 0$
 $$
\begin{aligned}
& f(t, X_t) = f(0, \xi)+\int_{0}^{t}   f_{x}(s, X_s) \Phi(s)  dZ_{s}  + \int_{0}^{t} f_{x}(s, X_s) \psi(s) d \beta(s)  \\
&  +\int_{0}^{t} f_{t}(s, X_s) d s +\frac{1}{2} \int_{0}^{t} 
\mbox{Tr}_{\Phi(s) Q_{Z}^{1 / 2}} \left( f_{x x}(s, X_s) \right)
\gamma(ds). 
\end{aligned}
$$
\end{corollary}

Now we discuss the general case. Assume that $Z=(Z_{t}: t \geq 0)$ is a cylindrical zero-mean square integrable  martingale such that for each $t \geq 0$ the mapping $Z_{t}: H \rightarrow L^{0}(\Omega,\mathcal{F},\Prob)$ is continuous.
As in Section \ref{subSecContAndDisconCMVM} we can decompose $Z$ as $Z_{t}=Z^{c}_{t}+Z^{d}_{t}$ where $Z^{c}$ is a cylindrical continuous zero-mean square integrable martingale and $Z^{d}$ is a cylindrical purely discontinuous zero-mean square integrable martingale. In this case, we can define three cylindrical martingale-valued measures:
$$ M(t,A)= Z_{t} \delta_{a}, \quad M^{c}(t,A)= Z^{c}_{t} \delta_{a}, \quad M^{d}(t,A)= Z^{d}_{t} \delta_{a}.$$
Clearly $M=M^{c}+M^{d}$ and since $Z^{c}$ and $Z^{d}$ are orthogonal, the corresponding families of intensity measures satisfies:
\begin{equation}\label{eqDecompoIntensityMeasures}
\nu_{h}=\nu^{c}_{h}+\nu^{d}_{h}.
\end{equation}
Assume that the family of intensity measures $(\nu_{h}: \norm{h}=1)$ of $M$ satisfies \eqref{eqContinuityIntensityMeasures}. 
We saw in the continuous paths setting, that the family $(\nu^{c}_{h}: \norm{h}=1)$ satisfies \eqref{eqDecompoIntensityMeasures}. Moreover, by \eqref{eqDecompoIntensityMeasures}, we conclude that \eqref{eqContinuityIntensityMeasures} also holds true for $(\nu^{d}_{h}: \norm{h}=1)$. Then by Proposition 5.7 in \cite{CCFM:SPDE} the families 
$(\nu_{h}: \norm{h}=1)$, $(\nu^{c}_{h}: \norm{h}=1)$ and $(\nu^{d}_{h}: \norm{h}=1)$
satisfy the sequential boundedness property. Then, as explained in Section \ref{subSecContAndDisconCMVM},  the quadratic variations $\operQuadraVari{M}$, $\operQuadraVari{M^{c}}$ and $\operQuadraVari{M^{d}}$ exist and are unique. 

As before, one can check that $\operQuadraVari{M}= \gamma \otimes \delta_{a}  $ $\Prob$-a.e., for $\gamma=\sup_{n \geq 1} \lambda_{\quadraVari{Z(h_{n})}}$ and so the condition \eqref{eqBoundedrangequadraticvariation} is trivially satisfied. Likewise we have 
$\operQuadraVari{M^{c}}= \gamma^{c} \otimes \delta_{a}  $ $\Prob$-a.e., for $\gamma^{c}=\sup_{n \geq 1} \lambda_{\quadraVari{Z^{c}(h_{n})}}$.
By the arguments in in Section \ref{subSecContAndDisconCMVM} we can show the existence of $Q_{M}$, $Q_{M^{c}}$ and $Q_{M^{d}}$. As before we can define
$Q_{Z}$ and $\Lambda^{2}(M,T)$ (and the analogs for $Z^{c}$ and $Z^{d})$. The equality in \eqref{eqCompatibleintegrals} remains valid for $M$ and $Z$ (respectively for $M^{c}$ an $Z^{c}$, and for $M^{d}$ an $Z^{d}$). As an application of Theorem \ref{theoItoFormulaVectorCase} we obtain the following It\^{o} formula for cylindrical square integrable martingales. We are not aware of any result of this nature in the literature. 
 
\begin{corollary}\label{coroItoFormulaDiscontinuousCylinSquaIntegMartingale}
Let $Z$, $Z^{c}$ and $Z^{d}$ as described above, and assume that the family of intensity measures $(\nu_{h}: \norm{h}=1)$ of $M$ satisfies \eqref{eqContinuityIntensityMeasures}.   
Let $\xi$, $\beta$, $\psi$ as in Definition \ref{defiItoProcess}. For $\Phi \in \Lambda^{2}(Z,T)$ consider  the It\^{o} process:  
$$X_t : = \xi + \int_0^t \psi(s) d\beta(s) + \int_0^t \Phi(s)\, dZ_{s}. $$
 Let $f:[0, \infty) \times G \rightarrow K$ be a function of class $C^{1,2}$ such that $f_{xx}$ is uniformly continuous on bounded subsets of $[0,\infty)\times G$. Then  almost surely we have for every $t\geq 0$
$$
\begin{aligned}
& f(t, X_t) = f(0, \xi)+ \int_{0}^{t} f_{t}(s, X_{s-}) d s +\int_{0}^{t} f_{x}(s, X_{s-}) \Phi(s)  dZ_{s} \\
& + \int_{0}^{t} f_{x}(s, X_{s-}) \psi(s) d \beta(s)   +\frac{1}{2} \int_{0}^{t} \!\!\int_U 
\mbox{Tr}_{\Phi(s) Q_{Z^{c}}^{1 / 2}} \left( f_{x x}(s, X_{s-}) \right) \gamma^{c}(ds)
 \\
 & +\sum_{0<s \leq t}\left[f\left(s,X_s\right)-f\left(s,X_{s-}\right)-f_x\left(s,X_{s-}\right)\left(\Delta X_{s} \right)\right]
\end{aligned}
$$
\end{corollary}

Cylindrical L\'evy processes were introduced in \cite{ApplebaumRiedle:2010}. Since then, several works introduced theories of stochastic integration with respect to these processes. Recently, in \cite{BodoRiedleTylb:2024}, they prove an It\^o formula for cylindrical standard $\alpha$-stable L\'evy processes for $1< \alpha <2$. However, we could not find in the literature a proof of an It\^o formula for cylindrical square integrable L\'evy processes. In the next example we apply Corollary \ref{coroItoFormulaDiscontinuousCylinSquaIntegMartingale} to prove such an It\^o formula with an emphasis on the components of the cylindrical L\'evy-It\^{o} decomposition. 

\begin{example}
Let $Z=(Z_{t}: t \geq 0)$ be a cylindrical zero-mean square integrable L\'{e}vy process in $X$, that is, for every $d \in \N$, $h_{1}, \cdots, h_{d} \in H$ the $\R^{d}$-valued stochastic process $(Z_{t}(h_{1}), \cdots, Z_{t}(h_{d}): t \geq 0)$ is a  L\'{e}vy process in $\R^{d}$, and $\Exp [Z_{t}(h) ] = 0 $ and $\Exp [ \abs{Z_{t}(h)}^{2} ]< \infty$ for every $t \geq 0$, $h \in H$. We further assume for such a $Z$ that the mapping $Z_{t}: H \rightarrow L^{0} \ProbSpace$ is continuous.   

It follows from Corollary 3.12 in \cite{ApplebaumRiedle:2010}, that $L$ can be decomposed into $L_{t}=W_{t}+P_{t}$. Here, $W=(W_{t}:t \geq 0)$ is a cylindrical Wiener process with covariance operator $Q_{W}:H \rightarrow H $ satisfying
$$ \Exp (\abs{ W_{t}(h)}^{2})=(t-s) \inner{Q_{W}h}{h}_{H}, \quad \forall \, h \in H, \, t \geq 0.$$ 
The cylindrical process $P=(P_{t}: t \geq 0)$ is independent of $W$, and is a cylindrical mean-zero  square integrable martingale given by 
$$ P_{t}(h)= \int_{\R \setminus \{0\}} \, \beta \widetilde{N}_{h} (t, d\beta), \quad \forall \, h \in H, $$
where $ N_{h}$ denotes the Poisson random measure associated with the L\'{e}vy process $(P_{t}(h): t \geq 0)$ and with corresponding L\'{e}vy measure $\lambda_{h}$. 
Theorem  2.7 in \cite{ApplebaumRiedle:2010} shows that there exists a cylindrical measure $\lambda$ on the cylindrical sets in $H$, called the cylindrical L\'{e}vy measure of $L$, such that for each $h \in H$ we have $\lambda_{h}=\lambda \circ \pi_{h}^{-1}$, where $\pi_{h}:H \rightarrow \R$, $\pi_{h}(y)=\inner{y}{h}_{H}$. In particular, we have 
$$ \Exp (\abs{ P_{t}(h)}^{2})=t \int_{H} \inner{y}{h}^{2} \lambda(dy), \quad \forall \, h \in H, \, t \geq 0.$$ 
It is easy to see that $Z$ is a cylindrical zero-mean square integrable  martingale. Moreover, $Z^{c}=W$ and $Z^{d}=P_{t}$.  Let $\tilde{Q}_{Z}: H\rightarrow H$ be given by 
$$\tilde{Q}_{Z}h=Q_{W}h+\int_{H} \inner{y}{h}^{2} \lambda(dy).$$
Then it is clear that  $\inner{\tilde{Q}_{Z}h_{1}}{h_{2}}_{H}=\Exp \left[ Z_{1}(h_{1}) Z_{1}(h_{2}) \right] $ $\forall h_{1}, h_{2} \in H$. The operator $\tilde{Q}_{Z}$ is known as the \emph{covariance operator} of $Z$ and it is linear, positive, symmetric and continuous. Moreover, since $Z(h)$ is a real-valued mean zero square integrable L\'evy process we have $\quadraVari{Z(h)}_{t} = t \Exp \left[\abs{Z_{1}(h)}^{2} \right]= t \inner{\tilde{Q}_Zh_{1}}{h_{2}}_{H}$.

Let $M$ be as in \eqref{eqCMVMContiyliMartingale}. By Example 5.14 in \cite{CCFM:SPDE} we have $\nu_{M}=\inner{\tilde{Q}_{Z}h}{h}_{H} (\mbox{Leb}\otimes \delta_{a})$ and the family of intensity measures $(\nu_{h}: \norm{h}=1)$ of $M$ satisfy \eqref{eqContinuityIntensityMeasures}. Moreover, we have $\operQuadraVari{M}=\norm{\tilde{Q}_{Z}}(\mbox{Leb}\otimes \delta_{a})$ and $Q_{Z}=\tilde{Q}_{Z}/\norm{\tilde{Q}_{Z}}$. One can likewise conclude that  $\operQuadraVari{M^{c}}=\norm{Q_{W}}(\mbox{Leb}\otimes \delta_{a})$ and $Q_{Z^{c}}=Q_{W}/\norm{Q_{W}}$. Observe that we have $\gamma^{c}= \norm{Q_{W}}\mbox{Leb}$.

 By following the arguments in Section 3 in  \cite{BodoRiedleTylb:2024} we can obtain an useful description of the sum of the jumps which involves the cylindrical L\'evy measure of $L$.  Consider the integer-valued random measure corresponding to the jumps of $X$:
$$\mu^{X}((0,t] \times B)= \sum_{0<s\leq t} \mathbbm{1}_{B}(\Delta X_{s}), \quad 0<t \leq T, \  B \in \mathcal{B}(H), \  0 \notin B. $$
$\mu^{X}$ is an optional $\sigma$-finite random measure on $\mathcal{B}([0,T]) \otimes \mathcal{B}(H)$, hence has a predictable compensator $\nu^{X}$ (see Section II.1 in \cite{JacodShiryaev:2003}). 
Moreover, from the definition of $\mu^{X}$ we have
\begin{multline}\label{eqSumJumpsAsIntegralCylLevy}
\sum_{0<s \leq t}\left[f\left(s,X_s\right)-f\left(s,X_{s-}\right)-f_x\left(s,X_{s-}\right)\left(\Delta X_{s} \right)\right]  \\
=\int_{0}^{t}\!\! \int_{H} \left[f\left(s,X_{s-} +h\right)-f\left(s,X_{s-}\right)-f_x\left(s,X_{s-}\right)\left(h \right)\right] \, \mu^{X}(ds,dh). 
\end{multline}
Observe that the jumps of $X$ correspond to those of the stochastic integral process $Y_{t}=\int_{0}^{t} \Phi(s)dP_{s}$. Hence, by a modification of the arguments used in the proof of Theorem 3.5 in \cite{BodoRiedleTylb:2024}, one can show that the predictable compensator $\nu^{X}$ of $\mu^{X}$ is characterized by:
\begin{equation}\label{eqPredCompenCylLevy}
 \nu^{X}((0,t]\times B)  =\int_{0}^{t} \left( \lambda \circ \Phi(s)^{-1}\right)(B)ds, \quad 0<t \leq T, \ B \in \mathcal{B}(H), \ 0 \notin B.
\end{equation}
Therefore, from Corollary \ref{coroItoFormulaDiscontinuousCylinSquaIntegMartingale}, \eqref{eqSumJumpsAsIntegralCylLevy} and \eqref{eqPredCompenCylLevy} we obtain the following expression for the It\^{o} formula for $X$:
$$
\begin{aligned}
& f(t, X_t) = f(0, \xi)+ \int_{0}^{t} f_{t}(s, X_{s-}) d s \\
& +\int_{0}^{t} f_{x}(s, X_{s-}) \Phi(s)  dW_{s} +\int_{0}^{t} f_{x}(s, X_{s-}) \Phi(s)  dP_{s} \\
& + \int_{0}^{t} f_{x}(s, X_{s-}) \psi(s) d \beta(s)   +\frac{1}{2} \int_{0}^{t} \!\!\int_H 
\mbox{Tr}_{\Phi(s) Q_{W}^{1 / 2}} \left( f_{x x}(s, X_{s-}) \right) ds
 \\
 & +\int_{0}^{t}\!\! \int_{H} \left[f\left(s,X_{s-} +h\right)-f\left(s,X_{s-}\right)-f_x\left(s,X_{s-}\right)\left(h \right)\right] \, (\mu^{X}-\nu^{X})(ds,dh) \\
  & +\int_{0}^{t}\!\! \int_{H} \left[f\left(s,X_{s-} +g\right)-f\left(s,X_{s-}\right)-f_x\left(s,X_{s-}\right)\left(g \right)\right] \, \left(\lambda \circ \Phi(s)^{-1}\right)(dg)ds
\end{aligned}
$$

\end{example}

\noindent \textbf{Acknowledgments}  This work was partially supported by The University of Costa Rica through the grant ``C3019- C\'{a}lculo Estoc\'{a}stico en Dimensi\'{o}n Infinita''. 

\noindent \textbf{Data Availability.} Data sharing not applicable to this article as no data sets were generated or analyzed during the current study. 

\noindent \textbf{Disclosure statement.} The authors report there are no competing interests to declare.

\section*{Declarations}

\noindent \textbf{Conflict of interest} The authors have no conflicts of interest to declare that are relevant to the content of this article.

\end{document}